\documentclass{article}
\usepackage{amsthm}
\usepackage{amsmath}
\usepackage[authoryear,round]{natbib}
\usepackage{hyperref}
\usepackage{hypernat}
\usepackage{amssymb}

\newcommand{\R}{{\mathbb R}}
\newcommand{\E}{{\mathbb E}}

\newcommand{\N}{{\mathbb N}}

\DeclareMathOperator{\NtG}{NtG}

\newtheorem{theorem}{Theorem}[section]
\newtheorem{proposition}[theorem]{Proposition}
\newtheorem{lemma}[theorem]{Lemma}

\newtheorem{remark}[theorem]{Remark}
\newtheorem*{remark*}{Remark}

\newtheorem*{condition*}{Condition}

\numberwithin{equation}{section}

\newcounter{rcnt}[section]
\renewcommand{\thercnt}{(\roman{rcnt})}
\newcommand{\rem}[1]{}

\begin{document}
\sloppy

\title{
	Admissibility of the usual confidence set \\
	for the mean of a univariate or bivariate normal population: 
	The unknown-variance case
}

\date{(revised manuscript)}

\author{Hannes Leeb (University of Vienna) and\\
	Paul Kabaila (La Trobe University)}

\maketitle

\begin{abstract}
In the Gaussian linear regression model (with unknown mean and variance),
we show that the standard confidence set
for one or two regression coefficients 
is admissible in the sense of \cite{Jos69a}.
This solves a  long-standing open problem in mathematical statistics,
and this has important implications on the performance
of modern inference procedures post-model-selection or post-shrinkage,
particularly in situations where the number of parameters is larger
than the sample size.
As a technical contribution of independent interest, 
we introduce a new class of
conjugate priors for the Gaussian location-scale model.
\end{abstract}



%

\section{Introduction and overview}
\label{S1}

Among the most widely used statistical methods are
inference procedures based on the Gaussian linear regression model
(with unknown mean and variance).
Studentized confidence intervals, in particular,
are a staple tool in applied analyses.
It is therefore important to know whether such simple
inference procedures are optimal.
This need not be the case. 
Indeed, any confidence interval for the variance
that is based only on the sample variance 
is sub-optimal and one can construct a confidence
interval that is uniformly shorter 
while having the same coverage probability;
see \cite{Gou91a} as well as  \cite{Tat59a}.
In this paper, we show that the usual standard  confidence
set for one or two regression coefficients can not be improved
in this way, because
this confidence set is
admissible in the sense of \cite{Jos69a}:
Consider a confidence procedure, for one or two regression coefficients,
whose measure (i.e., length or
area) is at least as small as that of the standard procedure,
and whose minimal coverage probability is at least as large as that
of the standard procedure.
Our results entail that such a procedure must coincide with the 
standard procedure almost everywhere;
see Theorem~\ref{t3} for details.
This extends earlier findings of \cite{Kab10a}. 
This also entails that no confidence interval, whose minimal coverage 
probability
equals that of the standard confidence set, can improve over
the standard set in terms of length.
[Valid confidence intervals that improve over the standard
interval in terms of length locally, e.g., at some
point or in some region of sample space, 
at the expense of increased length
elsewhere, are studied by 
\cite{Bro95a, Far08a, Kab09b}.]

Our results provide insights into the Stein phenomenon
for set-estimation
in the unknown-variance case. To explain, consider first
the known-variance case with 
independent observations from a $p$-dimensional 
normal distribution whose
mean vector is unknown and whose covariance matrix is the identity.
By sufficiency, this can be reduced to the Gaussian location model where 
$x\sim N(\mu, I_p)$ with unknown parameter $\mu\in\R^p$.
And recall the Stein phenomenon for point-estimation, i.e., the fact that
the standard estimator $x$ for the mean is admissible
with respect to squared error loss
if $p=1$ or $p=2$, 
while this standard estimator is inadmissible and can be
dominated by shrinkage estimators if $p\geq 3$; cf.
\cite{Ste56a} and \cite{Jam61a}.
It is well-known that the Stein phenomenon for point-estimation  
here also carries over to set-estimation: The standard
confidence set for the mean, i.e., a ball of fixed radius centered at $x$,
is admissible if $p=1$ or $p=2$, and this confidence set
can be  dominated if $p \geq 3$, e.g., by so-called re-centered
confidence sets; cf. \cite{Bro66a} and \cite{Jos67a,Jos69a}.

Now consider the corresponding unknown-variance case, i.e.,
independent observations from a $p$-dimensional normal distribution whose
mean is unknown and whose covariance matrix is an unknown
positive multiple of the identity. Assume that the number of observations,
that we denote by $n$, exceeds $p$.
By sufficiency, this can be reduced to the Gaussian location-scale model where
$x \sim N(\mu,\sigma^2 I_p)$ and $s /\sigma^2 \sim \chi^2_m$ independent
of $x$ (for $m=n-p \geq 1$).
Here, the data are $(x,s)$ and
the unknown parameters are $\mu\in \R^p$ and $\sigma^2>0$.
For point-estimation, it is easy to extend the Stein phenomenon from
the Gaussian location model to the Gaussian location-scale model.
It is thus tempting to conjecture, for set-estimation, 
that the Stein phenomenon 
can be extended in a similar fashion.
But in spite of strong numerical support reported in several of
the references that follow, an appropriate analytic result has
yet to be established.
\citet[p.~205]{Sal06a} notes that `{\em the confidence set with
unknown variance turned out to be a difficult problem that is open
for solution}.'
Partial results on dominance of the standard confidence set by
re-centered confidence sets in large dimensions are reported by 
\cite{	Ber80a,
	Car90a,
	Cas83a,
	Cas87c,
	Che88a,
	Hwa94a,
	Rob90a,
	Sam05a}.

Our results allow us to extend the results on the usual confidence set
in small dimensions from the known-variance case to the unknown-variance
case: We obtain that for $p=1$ and $p=2$ 
the standard confidence set 
in the Gaussian location-scale model
is admissible 
in the class
of all (possibly randomized) confidence sets 
in the sense
of \cite{Jos69a}; cf. Theorem~\ref{t1}.
In particular, we show that
any confidence set, that performs at least as well
as the standard confidence set in terms of minimal coverage probability
and in terms of measure, 
coincides with the standard confidence set almost surely.

Our findings also have important conceptual implications on the ongoing
development of valid inference procedures post-model-selection
or post-shrinkage. 
Consider first a Gaussian linear regression model with $d$ explanatory
variables, $n$ observations, and assume that, say,
a confidence interval is desired for a particular regression coefficient
or a linear contrast.  
[Similar considerations apply for two-dimensional quantities of interest.]
Assume for now that $n>d$.
Our results show that
non-standard confidence intervals obtained, e.g., 
through model selection or shrinkage,
that maintain a user-specified minimal coverage probability, can not be
smaller than
the corresponding standard interval.
And if a non-standard  confidence interval is smaller than the 
standard interval, then its minimal coverage probability must be smaller
than that of the standard interval.
This provides some vindication for 
recently proposed 
inference procedures
in this area  that are valid but  conservative,
in the sense that the resulting confidence sets 
have coverage probabilities at or above the nominal level,
and that these confidence sets  are larger
than the standard confidence set 
based on the overall model with positive probability. See
\cite{And09a, Poe09a, Poe10a, Sch14a}; and also the 
discussion in \cite{Lee14d}.
The situation becomes even more pronounced if the number of
parameters in the overall model exceeds the sample-size, i.e.,
in situations where $n < d$. Procedures relying on model selection or
shrinkage are particularly attractive in these situations. 
But here, any non-standard confidence interval can be compared
to the (infeasible) standard confidence interval that is constructed
from a sample of size $d+1$. In particular, we see that
non-standard confidence intervals here either 
have small minimal coverage probabilities
or  they must be quite large with positive probability.
In view of this, our results might also be seen as providing further support
to approaches to inference in `small-$n$-large-$d$' scenarios
that focus on non-standard quantities of interest instead of
the  underlying true parameter,
as in the works of 
\cite{Ber13a, Gen07a, Lee05a, Lee09a, Lee14b, Lee14a, Bac15a}.

The paper is organized as follows.
In Section~\ref{S2} we present our main findings, i.e.,
Theorems~\ref{t1} and Theorem~\ref{t3}, which
are both derived from a technical core result that we give in 
Proposition~\ref{t2}.
We in fact establish, for the standard procedure,
a version of admissibility that is stronger than
admissibility as considered by \cite{Jos69a}; see Remark~\ref{joshi}
(and also Remarks~\ref{Rproof}\ref{Rproof.2} and~\ref{Rstrong}).
The proof of Proposition~\ref{t2} is  lengthy and
is hence presented in a top-down fashion.
Section~\ref{S3} contains a high-level version of the
proof. 
Our main arguments rely on
a new class of conjugate priors for the Gaussian location-scale
model, which is presented in Section~\ref{S4}. 
Further technical details and proofs are collected in the 
supplementary material.
[In the supplementary material,
Appendix~\ref{AA} contains some technical remarks for Section~\ref{S2},
and the proofs of Proposition~\ref{PRD1} and Proposition~\ref{PRD2} are given
in Appendix~\ref{AB} and~\ref{AC}, respectively. Further auxiliary
results are presented in Appendix~\ref{AD}.]

\section{Main results}
\label{S2}

\subsection{The location-scale model}
\label{S2.1}
Throughout fix integers $p\geq 1$ and $m\geq 1$, and consider
independent random variables $x$ and $s$ with values in $\R^p$ and
$(0,\infty)$, respectively, so that  $x \sim N(\mu,\sigma^2 I_p)$
and so that $s/\sigma^2 \sim \chi^2_m$. 
The unknown parameters here are the mean $\mu\in\R^p$
and the variance $\sigma^2>0$.
In the following, we will write $E_{\mu,\sigma^2}[\cdots]$ for
the expectation of functions of $x$ and $s$
under the parameters $\mu$ and $\sigma^2$.

To study (possibly randomized) confidence sets for $\mu$
that depend on $x$ and $s$, we follow \cite{Jos69a} and
define a confidence procedure for $\mu$ as a measurable function
$\phi(x,s,\mu)$ from the product space $\R^p\times(0,\infty)\times \R^p$
to the unit  interval.
If $\phi$ takes on only the values $0$ and $1$, then it
can be interpreted as
a non-randomized confidence set with 
$\phi(x,s,\mu)=1$ if $\mu$ is included and
$\phi(x,s,\mu)=0$ otherwise;
the corresponding confidence set is $ C(x,s) = \{ \mu: \phi(x,s,\mu)=1\}$.
The standard procedure will be denoted by $\phi_0$ and is given
by $\phi_0(x,s,\mu)=1$ if $\|x-\mu\|^2 < c s/m$ and
$\phi_0(x,s,\mu)=0$ otherwise, for some $c>0$.
In general, a confidence procedure
$\phi$ can be interpreted as a randomized confidence set
with $\phi(x,s,\mu)$ equal to the conditional probability of
including $\mu$ given $x$ and $s$.
For any confidence procedure $\phi(x,s,\mu)$
and for fixed parameters $\mu$ and $\sigma^2$, note that
the coverage probability of $\phi$ is given by
$E_{\mu,\sigma^2}[ \phi(x,s,\mu)]$;
the (Lebesgue-) measure of $\phi$  is
denoted by $\upsilon(\phi(x,s,\cdot))$ and is
defined as $\upsilon(\phi(x,s,\cdot))= \int \phi(x,s,\mu) d \mu$.

\begin{theorem}\label{t1}
Fix $p \in \{1,2\}$ and $m\geq 1$, 
and recall that $\phi_0 = \phi_0(x,s,\mu)$ denotes
the standard confidence procedure.
Let $\phi_1=\phi_1(x,s,\mu)$ be any confidence procedure 
that performs at least as well as $\phi_0$ in terms of
expected measure conditional on $s$, and in terms of coverage probability; i.e,
$\phi_1$ satisfies
\begin{align}\label{t1.1}
E_{\mu,\sigma^2} \Big[ \upsilon( \phi_1(x,s,\cdot)) \;\Big\|\; s\Big] & 
\quad\leq \quad
E_{\mu,\sigma^2} \Big[ \upsilon( \phi_0(x,s,\cdot))\;\Big\|\; s\Big] 
\end{align}
almost everywhere and
\begin{align}\label{t1.2}
E_{\mu,\sigma^2} \Big[\phi_1(x,s,\mu)\Big] & 
\quad \geq \quad 
E_{\mu,\sigma^2} \Big[\phi_0(x,s,\mu)\Big]
\end{align}
for each $(\mu,\sigma^2) \in\R^p \times (0,\infty)$.
Then $\phi_1 = \phi_0$ almost everywhere.
\end{theorem}

\begin{remark}\normalfont\label{joshi}
Following \cite{Jos69a}, the standard procedure $\phi_0$ is admissible
if any other procedure $\phi_1$ that satisfies
$$
 \upsilon( \phi_1(x,s,\cdot))  
\quad\leq \quad
 \upsilon( \phi_0(x,s,\cdot))
$$
almost everywhere, and that satisfies \eqref{t1.2} for each
$(\mu,\sigma^2)$, is such that $\phi_1=\phi_0$ almost everywhere.
Theorem~\ref{t1} entails that the standard procedure
is admissible, because \eqref{t1.1} follows from the relation in the
preceding display. 
In view of this,
our theorem delivers a stronger
version of admissibility, because condition \eqref{t1.1} is weaker
than the condition expressed in the preceding display.
The same applies, mutatis mutandis,  to
the standard procedure in the linear regression model,
which is discussed in the following section; cf. Theorem~\ref{t3}.
\end{remark}

Both Theorem~\ref{t1} and Theorem~\ref{t3}, which is presented in
the next section, are consequences of the following technical result, where
we consider a function $\phi_1$ that also depends on $\sigma^2$,
i.e., $\phi_1 = \phi_1(x,s,\mu,\sigma^2)$, so that $\phi_1$ is
a measurable function from 
$\R^p\times (0,\infty) \times \R^p\times (0,\infty)$ 
to $[0,1]$. 
Because of its dependence on $\sigma^2$, such a function $\phi_1$ need
not correspond to a (feasible) confidence procedure.
Of course, the standard procedure $\phi_0$ can also be viewed as
a function $\phi_0(x,s,\mu,\sigma^2)$ (that is constant in its last
argument).
Similarly to before, we set, e.g., 
$\upsilon(\phi_1(x,s,\cdot,\sigma^2)) = \int \phi_1(x,s,\mu,\sigma^2) d\mu$.

\begin{proposition}\label{t2}
Fix $p \in \{1,2\}$ and $m\geq 1$,
and let $\phi_0 = \phi_0(x,s,\mu,\sigma^2)$ 
and $\phi_1 = \phi_1(x,s,\mu,\sigma^2)$ 
be as in the
preceding paragraph.
If $\phi_1$ satisfies
\begin{align}\label{t2.1}
E_{\mu,\sigma^2} \Big[ 
\upsilon( \phi_1(x,s,\cdot,\sigma^2)) \;\Big\|\;s \Big] & 
\quad\leq \quad
E_{\mu,\sigma^2} \Big[ 
\upsilon( \phi_0(x,s,\cdot,\sigma^2))\;\Big\|\;s \Big] 
\end{align}
almost everywhere and
\begin{align}
\label{t2.2}
E_{\mu,\sigma^2} \Big[\phi_1(x,s,\mu,\sigma^2)\Big] & 
\quad \geq \quad 
E_{\mu,\sigma^2} \Big[\phi_0(x,s,\mu,\sigma^2)\Big]
\end{align}
for each $(\mu,\sigma^2) \in\R^p \times (0,\infty)$,
then $\phi_1 = \phi_0$ almost everywhere.
\end{proposition}

Theorem~\ref{t1} obviously is a special case of Proposition~\ref{t2}.
We will see that Proposition~\ref{t2} can also be used to deal
with the Gaussian linear regression model.

\subsection{The linear regression model}
\label{S2.2}

Consider the linear regression model $y = Z \beta + u$, 
where $Z$ is a fixed $n\times d$ matrix
of rank $d<n$ and $u \sim N(0,\sigma^2 I_n)$. The unknown parameters
here are $\beta \in \R^d$ and $\sigma^2>0$.
Write $\hat{\beta}$ and $\hat{\sigma}^2$ for the usual unbiased
estimators for $\beta$ and $\sigma^2$, i.e.,
$\hat{\beta} = (Z'Z)^{-1} Z' y$ and
$\hat{\sigma}^2 = \| y - Z \hat{\beta}\|^2/(n-d)$.
The expectation of functions of $\hat{\beta}$ and $\hat{\sigma}^2$
under the true parameters $\beta$ and $\sigma^2$ will be denoted
by $\E_{\beta,\sigma^2}[\cdots]$.

Fix an integer $p \leq d$ and consider the standard confidence set
for the first $p$ components of $\beta$, which we denote by $\beta_{(p)}$.
To this end, partition 
$\beta$ as $\beta' = (\beta_{(p)}', \beta_{(\neg p)}')$
and partition 
$\hat{\beta}$ conformably as
$\hat{\beta}' = (\hat{\beta}_{(p)}', \hat{\beta}_{(\neg p)}')$
in case $p<d$; in case $p=d$, we set  
$\hat{\beta}_{(p)} = \hat{\beta}$ and
$\beta_{(p)} = \beta$.
Finally, write 
$\sigma^2 S_{(p)}$ for  the  
covariance matrix of
$\hat{\beta}_{(p)}$, i.e., $S_{(p)}$ denotes the
leading $p\times p$ submatrix of $(Z'Z)^{-1}$.
As (possibly randomized) confidence procedures, we consider
measurable functions $\varphi(\hat{\beta}, \hat{\sigma}^2, \beta_{(p)})$
from the product space $\R^d\times (0,\infty)\times \R^p$
to the unit interval.
The standard confidence procedure here will be denoted by
$\varphi_0$ and is defined by
$\varphi_0(\hat{\beta}, \hat{\sigma}^2, \beta_{(p)}) = 1$ if
$\|S_{(p)}^{-1/2}(\beta_{(p)} - \hat{\beta}_{(p)})\|^2 < c 
\hat{\sigma}^2$
and $\varphi_0(\hat{\beta}, \hat{\sigma}^2, \beta_{(p)}) = 0$ otherwise.
For any confidence procedure 
$\varphi(\hat{\beta},\hat{\sigma}^2,\beta_{(p)})$ 
and for fixed parameters $\beta$ and $\sigma^2$,
the coverage probability of $\varphi$ is given by
$\E_{\beta,\sigma^2}[ \varphi(\hat{\beta},\hat{\sigma}^2,\beta_{(p)})]$;
and the measure of $\varphi$ is given by
$\upsilon(\varphi(\hat{\beta},\hat{\sigma}^2,\cdot)) =
\int \varphi(\hat{\beta},\hat{\sigma}^2,\beta_{(p)}) d \beta_{(p)}$.

\begin{theorem}\label{t3}
Fix $p \in \{1,2\}$ and recall that 
$\varphi_0 = \varphi_0(\hat{\beta},\hat{\sigma}^2, \beta_{(p)})$ 
denotes the standard confidence procedure. Let 
$\varphi_1 = \varphi_1(\hat{\beta},\hat{\sigma}^2, \beta_{(p)})$ be any
confidence procedure that performs at least as well as $\varphi_0$
in terms of expected measure conditional on $\hat{\sigma}^2$,
and in terms of coverage probability;
that is, $\varphi_1$ satisfies
\begin{align}\label{t3.1}
\E_{\beta,\sigma^2}\Big[
\upsilon( 
	\varphi_1(\hat{\beta},\hat{\sigma}^2,\cdot)) \;\Big\|\; \hat{\sigma}^2
		\Big]& 
\quad\leq \quad
\E_{\beta,\sigma^2}\Big[
\upsilon( 
	\varphi_0(\hat{\beta},\hat{\sigma}^2,\cdot)) \;\Big\|\; \hat{\sigma}^2
		\Big]
\end{align}
almost everywhere and
\begin{align}\label{t3.2}
\E_{\beta,\sigma^2} \Big[\varphi_1(\hat{\beta},\hat{\sigma}^2,\beta_{(p)})
	\Big] & 
\quad \geq \quad 
\E_{\beta,\sigma^2} \Big[\varphi_0(\hat{\beta},\hat{\sigma}^2,\beta_{(p)})
	\Big]
\end{align}
for each $\beta\in \R^d$ and each $\sigma^2 >0$. 
Then $\varphi_1 = \varphi_0$ almost everywhere.
\end{theorem}

\begin{proof}
Consider first the case where $d>p$, and
fix $\beta_{(\neg p)}$ for the moment.
The standard procedure $\varphi_0$ depends on $\hat{\beta}$
only through $\hat{\beta}_{(p)}$, i.e.,
$\varphi_0 = 
\varphi_0(\hat{\beta}_{(p)}, \hat{\sigma}^2, \beta_{(p)})$.
Consider now the conditional mean of $\varphi_1$ given $\hat{\beta}_{(p)}$
and $\hat{\sigma}^2$,
which is a function of $\hat{\beta}_{(p)}$
and $\hat{\sigma}^2$, and which also depends on
$\beta_{(p)}$, $\beta_{(\neg p)}$ and $\sigma^2$
(because the law of $\hat{\beta}_{(\neg p)}$ given 
$\hat{\beta}_{(p)}$ and $\hat{\sigma}^2$ depends on these
parameters). Set
\begin{equation}\label{pt3.1}
\varphi_{1|\beta_{(\neg p)}}
	(\hat{\beta}_{(p)}, \hat{\sigma}^2, \beta_{(p)}, \sigma^2)
\quad=\quad
\E_{\beta,\sigma^2}\left[\left.
	\varphi_1
	(\hat{\beta}, \hat{\sigma}^2, \beta_{(p)})
	\right\| \hat{\beta}_{(p)}, \hat{\sigma}^2\right].
\end{equation}
Note that this conditional mean also depends on $\sigma^2$;
the implications of this are further discussed in Remark~\ref{r1}.
It is not difficult to show that $\varphi_{1|\beta_{(\neg p)}}$
is a measurable function from 
$\R^p\times (0,\infty)\times \R^p\times (0,\infty)$
to the unit interval; see Remark~\ref{TR1} for details.

Set $x =  S_{(p)}^{-1/2}\hat{\beta}_{(p)}$,
$\mu = S_{(p)}^{-1/2}\beta_{(p)}$ and $s = (n-d) \hat{\sigma}^2$,
and
define $\phi_0$ and $\phi_1$ by
\begin{align*}
\phi_0(x,s,\mu,\sigma^2)& \quad=\quad
	\varphi_{0}\left(
	\hat{\beta}_{(p)}, \hat{\sigma}^2, \beta_{(p)}
	\right) \qquad
	\text{ and}
\\
\phi_{1}(x,s,\mu,\sigma^2)& \quad=\quad
	\varphi_{1|\beta_{(\neg p)}}\left(
	\hat{\beta}_{(p)}, \hat{\sigma}^2, \beta_{(p)}, \sigma^2
	\right).
\end{align*}
It is easy to see
that $\phi_0$ and $\phi_1$ satisfy
the assumptions of Proposition~\ref{t2}; cf. Remark~\ref{TR2}.

Proposition~\ref{t2} entails that $\phi_0 = \phi_1$ almost everywhere or,
equivalently, that $\varphi_0 = \varphi_{1|\beta_{(\neg p)}}$ almost
everywhere. More precisely, we have 
\begin{equation}\label{pt3.2}
\varphi_0(\hat{\beta}_{(p)}, \hat{\sigma}^2, \beta_{(p)}) 
\quad=\quad
\varphi_{1|\beta_{(\neg p)}}(\hat{\beta}_{(p)}, \hat{\sigma}^2, \beta_{(p)},
	\sigma^2)
\end{equation}
for almost all $(\hat{\beta}_{(p)}, \hat{\sigma}^2, \beta_{(p)},\sigma^2)
\in \R^p\times (0,\infty)\times \R^p\times (0,\infty)$.
Dropping the assumption that $\beta_{(\neg p})$ is fixed we can conclude
that the relation in the preceding display holds
for almost all $(\hat{\beta}_{(p)}, \hat{\sigma}^2, \beta,\sigma^2) 
\in \R^p\times (0,\infty) \times \R^d\times (0,\infty)$.

Fix $\hat{\beta}_{(p)}$, $\hat{\sigma}^2$, $\beta_{(p)}$, and $\sigma^2$
for the moment, so that the \eqref{pt3.2} holds for almost all
$\beta_{(\neg p)}$, and recall that $\varphi_{1|\beta_{(\neg p)}}$ is 
the conditional mean of a function of $\hat{\beta}_{(\neg p)}$ given 
$\hat{\beta}_{(p)}$ and $\hat{\sigma}^2$.
Because the conditional distributions of $\hat{\beta}_{(\neg p)}$
given $\hat{\beta}_{(p)}$ and $\hat{\sigma}^2$
can be viewed as a full-rank exponential family parameterized
by $\beta_{(\neg p)}$, and because $\hat{\beta}_{(\neg p)}$ is a
complete statistic for that family, it follows from \eqref{pt3.2} that
\begin{equation}\label{pt3.3}
\varphi_0(\hat{\beta}_{(p)}, \hat{\sigma}^2, \beta_{(p)})
\quad=\quad
\varphi_1(\hat{\beta}, \hat{\sigma}^2, \beta_{(p)})
\end{equation}
holds 
for almost all $(\hat{\beta}, \hat{\sigma}^2, \beta_{(p)},\sigma^2 )
\in \R^d \times (0,\infty) \times \R^p\times(0,\infty)$. For details, see 
Remark~\ref{TR3}.
This completes the proof in case $d>p$.

In the case where $d=p$, we argue as in the second paragraph and in the
first two sentences of the the third
paragraph of the proof, with $\hat{\beta}_{(p)}$ and $\beta_{(p)}$ 
now set 
equal to $\hat{\beta}$ and $\beta$, respectively, with
$S_{(p)}= (Z'Z)^{-1}$, and with 
$\varphi_{1|\beta_{(\neg p)}}(\hat{\beta}_{(p)}, \hat{\sigma}^2,
\beta_{(p)}, \sigma) = 
\varphi_1(\hat{\beta}, \hat{\sigma}^2, \beta)$.
\end{proof}

\begin{remark}\normalfont
\label{r1}
In the known-variance case, corresponding results for the 
linear regression model follow directly from corresponding results
for the location model; in other words, a known-variance
version of Theorem~\ref{t3}  follows from a known-variance
version of Theorem~\ref{t1} (by arguing as in Section 4 of \citealt{Kab10a}).
This is not so when the variance is unknown, and a more
general result, namely Proposition~\ref{t2}, is needed here.
See also the first paragraph in the proof of Theorem~\ref{t3}.
\end{remark}

\section{Proof of Proposition~\ref{t2}}
\label{S3}

We re-parameterize the variance as
$\sigma^2 = 1/\lambda$, so that the procedures in Proposition~\ref{t2}
are $\phi_0 = \phi_0(x,s,\mu,\lambda)$ and
$\phi_1 = \phi_1(x,s,\mu,\lambda)$, and so that
expectations under the parameters $\mu$ and $\lambda$ are written
as $E_{\mu,\lambda}[\cdots]$.
Following \cite{Bly51a},  we consider a hierarchical model with
a certain prior on the parameters $\mu\in \R^p$ and $\lambda>0$,
so that $x$ and $s$ are distributed as described earlier
conditional on $(\mu,\lambda)$.
As the prior, we use a particular instance of 
a new conjugate prior for the Gaussian location-scale model.
This new conjugate prior, which we call the
normal-truncated-gamma prior, is introduced in
Section~\ref{S4.1} in its general form
along with some basic properties of that prior.
The particular instance of the normal-truncated-gamma prior,
which we use in the following, is denoted by 
$\NtG(p,0,\kappa, -p/2,0,\epsilon)$
in the notation of Section~\ref{S4.1}.
Using this prior, we obtain a joint density 
for $x,s,\mu,\lambda$ that depends on the hyper-parameters $\kappa>0$
and $\epsilon>0$,
that we denote by $p_{\kappa,\epsilon}(x,s,\mu,\lambda)$,
and that we factorize as 
$$
	p_{\kappa,\epsilon}(x,s,\mu,\lambda) \quad =\quad
	p_{\kappa,\epsilon}(x,s)
	\;\;
	p_{\kappa,\epsilon}(\lambda\|x,s) 
	\; \;
	p_{\kappa,\epsilon}(\mu\|x,s,\lambda)  
	.
$$
In the preceding display, $p_{\kappa,\epsilon}(x,s)$ 
denotes the marginal density
of $(x,s)$, $p_{\kappa,\epsilon}(\lambda\|x,s)$ 
denotes the conditional marginal
density of $\lambda$ given $(x,s)$, and 
$p_{\kappa,\epsilon}(\mu\|x,s,\lambda)$
denotes the conditional density of $\mu$ given $(x,s,\lambda)$,
all under the $\NtG(p,0,\kappa, -p/2,0,\epsilon)$-prior.
Lastly, set $p_{\kappa,\epsilon}(\mu,\lambda\|x,s) = 
p_{\kappa,\epsilon}(\lambda\|x,s) p_{\kappa,\epsilon}(\mu\|x,s,\lambda)$.
It will always be clear from the context how expressions
like $p_{\kappa,\epsilon}(\cdots)$ are to be interpreted.
Explicit expressions for the densities in the preceding display,
and for related quantities,
are given in Section~\ref{S4.2}. For the level of discussion
maintained here, it suffices to point out that
$p_{\kappa,\epsilon}(\mu\|x,s,\lambda)$ 
is the density of the $N(\mu_\kappa,(\lambda(1+\kappa))^{-1}I_p)$-distribution
at $\mu$. This density is does not depend on $\epsilon$ and $s$,
is spherically symmetric in $\mu$
around $\mu_\kappa = x/(1+\kappa)$, is maximized at $\mu=\mu_\kappa$,
and decreases  as $\|\mu-\mu_\kappa\|$ increases.
We will also write $p_{\kappa,\epsilon}(\mu\|x,s,\lambda)$ as
$$
p_{\kappa,\epsilon}(\mu\| x,s,\lambda)\quad=\quad 
r_\kappa(\|\mu-\mu_\kappa\|^2 |  \lambda)
$$
in the following.

Consider an improper (un-normed) version of the prior
that we denote by $q_{\kappa,\epsilon}(\mu,\lambda)$
and that is of the form
$q_{\kappa,\epsilon}(\mu,\lambda) = 
	K_{\kappa,\epsilon} p_{\kappa,\epsilon}(\mu,\lambda)$
for 
$$
K_{\kappa,\epsilon} \quad=\quad
	\frac{2}{p}\;
	\Gamma\left(\frac{m}{2}\right)
	\left(\frac{2 \pi}{\epsilon}\;
	\frac{1+\kappa}{\kappa}\right)^\frac{p}{2}.
$$
The constant $K_{\kappa,\epsilon}$ is such
that $q_{\kappa,\epsilon}(\mu,\lambda)$ converges, as $\kappa\to 0$ while
$\epsilon>0$ is fixed,
to the density of a $\sigma$-finite measure on $\R^p\times (0,\infty)$,
that we denote by $q_{0,\epsilon}(\mu,\lambda)$.
Again, explicit formulae are given in Section~\ref{S4.2}.
When $\epsilon>0$ is fixed and $\kappa\to 0$, note that
$K_{\kappa,\epsilon}$ is of the order $O(\kappa^{-p/2})$.
Also consider the function
$$
	q_{\kappa,\epsilon}(x,s,\mu,\lambda) \quad=\quad 
	K_{\kappa,\epsilon}
	p_{\kappa,\epsilon}(x,s,\mu,\lambda).
$$
Obviously, we can factorize 
$q_{\kappa,\epsilon}(x,s,\mu,\lambda)$ as
$$
q_{\kappa,\epsilon}(x,s,\mu,\lambda) \quad=\quad 
q_{\kappa,\epsilon}(x,s) \;\; 
p_{\kappa,\epsilon}(\lambda\|x,s) \;\; p_{\kappa,\epsilon}(\mu\|x,s,\lambda)
$$
for 
$q_{\kappa,\epsilon}(x,s) = K_{\kappa,\epsilon} p_{\kappa,\epsilon}(x,s)$.
While the proper prior $p_{\kappa,\epsilon}(\mu,\lambda)$ is defined only
for $\kappa>0$ and $\epsilon>0$, 
the improper prior $q_{\kappa,\epsilon}(\mu,\lambda)$,
the  function $q_{\kappa,\epsilon}(x,s)$, 
the conditional densities 
$p_{\kappa,\epsilon}(\lambda\|x,s)$ 
and
$p_{\kappa,\epsilon}(\mu\|x,s,\lambda)$ or, equivalently,
$r_\kappa(\|\mu-\mu_\kappa\|^2|\lambda)$,
as well $\mu_\kappa$ are well-defined
also in case $\kappa=0$ and $\epsilon>0$ 
(by the formulae in Section~\ref{S4.2} and in view of
the preceding conventions).
In particular, $\mu_0=x$.

In the following, posterior means will be denoted by
expressions of the form $P_{\kappa,\epsilon}(\cdots\|x,s)$, i.e., 
$P_{\kappa,\epsilon}(\cdots\|x,s) =
	\iint\cdots p_{\kappa,\epsilon}(\mu,\lambda\|x,s)\,d\mu\,d\lambda$.
The corresponding proper and improper prior means 
are $P_{\kappa,\epsilon}(\cdots) = 
	\iint P_{\kappa,\epsilon}(\cdots\|x,s) 
		p_{\kappa,\epsilon}(x,s) \,d x\,d s$
(in case $\kappa>0$)
and $Q_{\kappa,\epsilon}(\cdots) = 
	\iint P_{\kappa,\epsilon}(\cdots\|x,s) 
	q_{\kappa,\epsilon}(x,s) \,d x\,d s$,
respectively.
Similar notation will be used to denote other conditional means
like $P_{\kappa,\epsilon}(\cdots \|\mu,\lambda)$
and $P_{\kappa,\epsilon}(\cdots\|s,\mu,\lambda)$.
Note that the latter two expressions coincide with (frequentist)
means and conditional means given $s$, respectively,
in the notation of Section~\ref{S2}; i.e.,
$P_{\kappa,\epsilon}(\cdots \|\mu,\lambda)
	= E_{\mu,\lambda}[ \cdots]$ 
and
$P_{\kappa,\epsilon}(\cdots \|s, \mu,\lambda)
	= E_{\mu,\lambda}[ \cdots\| s]$.
We use the symbols $x$, $s$, $\mu$ and
$\lambda$ do denote both random variables and integration variables.
It will always be clear from the context how these symbols are to be
interpreted.

Let $\phi = \phi(x,s,\mu,\lambda)$  be a function
from $\R^p\times (0,\infty)\times \R^p\times (0,\infty)$ to the
unit interval that is measurable, and set
$\upsilon( \phi(x,s,\cdot,\lambda)) = \int \phi(x,s,\mu,\lambda) d \mu$.
For the standard procedure, note that $\upsilon(\phi_0(x,s,\cdot,\lambda)$,
i.e., the volume of a ball of radius $(c s/m)^{1/2}$ in $\R^p$, 
equals $(\pi c s/m)^{p/2}/ \Gamma(p/2+1)$.
For each $\kappa\geq 0$ and $\epsilon>0$, we consider the  
loss-function 
$$
L_\kappa (\phi) \quad=\quad
	r_\kappa( c s/m|\lambda) 
	\upsilon(\phi(x,s,\cdot,\lambda)) - \phi(x,s,\mu,\lambda),
$$
where $c$ governs the diameter of the standard confidence set,
i.e., $\phi_0(x,s,\mu)$ equals one if $\|\mu-x\|^2 < c s/m$ and 
zero otherwise.  
Note that $L_\kappa(\phi_0)$ is equal to the volume of the standard
confidence set, weighted by $r_\kappa(c s/m|\lambda)$, minus the indicator
on the event that $\mu$ is covered.
Of course, $L_\kappa(\phi)$
depends on $x$, $s$, $\lambda$ and $\mu$, but this dependence is not shown
explicitly in the notation for the sake of brevity.
If $\kappa>0$, the corresponding risk is
$P_{\kappa,\epsilon} (L_\kappa (\phi))$.
This risk is well-defined for each $\kappa>0$ and $\epsilon>0$,
and it satisfies 
$P_{\kappa,\epsilon} (L_\kappa (\phi))\geq -1$, 
because $L_\kappa(\phi)$
is the difference of two non-negative functions where the second one,
namely $\phi(x,s,\mu,\lambda)$, is bounded from above by $1$.

\begin{proposition} \label{bayesprocedure}
Fix $p\geq 1$ and $m\geq 1$.
For each $\kappa\geq 0$, each $\epsilon>0$ and
each $(x,s) \in \R^p\times(0,\infty)$,
the posterior risk 
$P_{\kappa,\epsilon}(L_\kappa(\phi)\|x,s)$ is 
minimized for $\phi=\phi_\kappa$, where $\phi_\kappa(x,s,\mu,\lambda) = 1$
if $\|\mu-\mu_\kappa\|^2 < c s/m$ and $\phi_\kappa(x,s,\mu,\lambda)=0$ 
otherwise.
In particular, the minimizer $\phi_\kappa$ is independent of $\lambda$
and $\epsilon$.
If $\kappa>0$, then $\phi_\kappa$ is a (proper) Bayes procedure 
whose risk satisfies 
$-1 \leq P_{\kappa,\epsilon}(L_\kappa(\phi_\kappa)) \leq 0$.
Moreover, $\phi_0$ is a generalized  Bayes procedure with 
$Q_{0,\epsilon}(L_\kappa(\phi_0))\leq 0$.
\end{proposition}

\begin{proof}
The posterior risk 
$P_{\kappa,\epsilon}(L_\kappa(\phi)\|x,s)$ can be written as
the integral over $\lambda \in (0,\infty)$ of the product of 
$p_{\kappa,\epsilon}(\lambda\|x,s)$ and
\begin{align*}
	&\int 
	\left( r_\kappa(c s/m|\lambda) 
		\upsilon(\phi(x,s,\cdot,\lambda))  
		- \phi(x,s,\mu,\lambda)\right) 
	\, p_{\kappa,\epsilon}(\mu\|x,s,\lambda) \; d \mu.
\end{align*}
Because $r_\kappa(c s/m|\lambda)$ does not depend
on $\mu$, the integral in the preceding display can also be written as
\begin{align*}
	& r_\kappa( c s/m|\lambda) 
		\int \phi(x,s,\mu,\lambda) \;d \mu - 
	\int \phi(x,s,\mu,\lambda) \, p_{\kappa,\epsilon}(\mu\|x,s,\lambda)
	\;d\mu
	\\
	&\quad=\quad
	\int \left(r_\kappa(c s/m|\lambda) 
		- p_{\kappa,\epsilon}(\mu\|x,s,\lambda)\right)
	\, \phi(x,s,\mu,\lambda) \; d\mu.
\end{align*}
In the preceding display,
the integral on the right-hand side of the equality is minimized
by taking $\phi(x,s,\mu,\lambda) = 1$ whenever 
$r_\kappa(c s/m|\lambda) < 
	p_{\kappa,\epsilon}(\mu\|x,s,\lambda)$ and zero otherwise.
Recalling that 
$p_{\kappa,\epsilon}(\mu\|x,s,\lambda) = 
	r_\kappa(\|\mu-\mu_\kappa\|^2| \lambda)$ 
is decreasing as $\|\mu-\mu_\kappa\|$ increases, we see that
a minimizing procedure is obtained by setting
$\phi(x,s,\mu,\lambda)=1$ whenever $\|\mu-\mu_\kappa\|^2 < c s/m$
and zero otherwise, i.e., by setting $\phi = \phi_\kappa$.
By construction, we have 
$P_{\kappa,\epsilon}( L_\kappa(\phi_\kappa) \|x,s) \leq 0$.

Finally, $Q_{\kappa,\epsilon}(L_\kappa(\phi_\kappa))$,
and also $P_{\kappa,\epsilon}(L_\kappa(\phi_\kappa))$ 
in case $\kappa>0$,
is bounded from above by zero, because the 
posterior risk
is non-positive by construction.
The lower bound on $P_{\kappa,\epsilon}(L_\kappa(\phi_\kappa))$ 
in case $\kappa>0$
has already been derived in the discussion leading up to 
Proposition~\ref{bayesprocedure}.
\end{proof}

We now compare the standard procedure $\phi_0$ and the Bayes procedure
$\phi_\kappa$ in terms of risk. [The proofs of the following two propositions
are more technical and therefore relegated to the Appendix.]

\begin{proposition} \label{PRD1}
Fix $p\geq 1$, $m\geq 1$, as well as $\kappa>0$ and $\epsilon>0$.
Then the risk difference satisfies
$$
P_{\kappa,\epsilon}( L_\kappa(\phi_0)) - 
	P_{\kappa,\epsilon}(L_\kappa(\phi_\kappa)) 
	\quad=\quad  
	F_{p,m}\bigg(\frac{c}{p}(1+\kappa)\bigg) 
	-
	F_{p,m}\bigg(\frac{c}{p}\bigg),
$$
where $F_{p,m}(\cdot)$ denotes the cumulative distribution function
of the $F$-distribution with $p$ and $m$ degrees of freedom.
\end{proposition}

Note that the risk difference does not depend on $\epsilon$.
For $\kappa\to 0$,
we see that the risk difference between the standard procedure $\phi_0$
and the Bayes procedure $\phi_\kappa$ under $P_{\kappa,\epsilon}$ 
is $O(\kappa)$.
Compared to the scaling constant of $p_{\kappa,\epsilon}(x,s) = 
K_{\kappa,\epsilon}^{-1} q_{\kappa,\epsilon}(x,s)$, which is
of the order $O(\kappa^{p/2})$, the risk difference converges to zero
at a faster rate if $p=1$, and the same rate if $p=2$, and at a slower
rate if $p>2$.
In other words, for fixed $\epsilon>0$ and as $\kappa\to 0$,
the risk difference of
$\phi_0$ and $\phi_\kappa$ under $Q_{\kappa,\epsilon}$ 
converges to zero in case $p=1$,
to a constant in case $p=2$, and to infinity in case $p\geq 3$.
This will allow us to derive the conclusion of Proposition~\ref{t2}.

Now assume that the assumptions of Proposition~\ref{t2} are satisfied.
In particular, $p$ equals $1$ or $2$, $m$ is a fixed integer,
and the confidence procedure $\phi_1$ performs at least as well as the 
standard procedure $\phi_0$ in terms of measure (almost surely) and 
coverage probability (everywhere in parameter space).
Fix $\kappa>0$ and $\epsilon>0$ for the moment.
Because 
$$
P_{\kappa,\epsilon}( \phi_1(x,s,\mu,\lambda)\|\mu,\lambda)
\quad=\quad
E_{\mu,\lambda} [\phi_1(x,s,\mu,\lambda)]
$$
and 
$$
P_{\kappa,\epsilon}( \upsilon( \phi_1(x,s,\cdot,\lambda)) \|s, \mu,\lambda)
\quad=\quad
E_{\mu,\lambda} [\upsilon(\phi_1(x,s,\cdot,\lambda)) \|s],
$$
and because the relations in the two preceding displays also hold
with $\phi_0$ replacing $\phi_1$, it follows from
\eqref{t2.1} and \eqref{t2.2} that
\begin{equation}
\label{key}
Q_{\kappa,\epsilon}( 
	L_\kappa( \phi_1)) - 
Q_{\kappa,\epsilon}( L_\kappa( \phi_0)) \quad \leq \quad 0,
\end{equation}
in view of the definition of the loss
$L_\kappa$ (note that $r_\kappa(c s/m|\lambda)$ is a function of $s$.
If $\phi_1(x,s,\mu,\lambda) = 
\phi_0(x,s,\mu,\lambda)$ holds Lebesgue almost-everywhere,
then the conclusion of Proposition~\ref{t2} holds.
If $\phi_1(x,s,\mu,\lambda) \neq \phi_0(x,s,\mu,\lambda)$ 
on a set of positive
Lebesgue measure, then the next result entails that the
left-hand side of the preceding display must be positive for
certain $\kappa>0$ and $\epsilon>0$.
This leads to a contradiction and completes the proof of Proposition~\ref{t2}.

\begin{proposition} \label{PRD2}
Assume that Proposition~\ref{t2} applies.
If the confidence procedure $\phi_1$ 
is such that $\phi_1(x,s,\mu,\lambda) \neq \phi_0(x,s,\mu,\lambda)$ 
on a set 
of positive Lebesgue measure,
then
$$
Q_{\kappa,\epsilon}( L_\kappa(\phi_1)) -
Q_{\kappa,\epsilon}( L_\kappa(\phi_0))
\quad > \quad 0
$$
whenever $\kappa>0$ and $\epsilon>0$ 
are sufficiently close to zero.
\end{proposition}

\begin{remark}\normalfont
\label{Rproof}
\begin{list}{\thercnt}{
        \usecounter{rcnt}
        \setlength\itemindent{5pt}
        \setlength\leftmargin{0pt}
        \setlength\partopsep{0pt}
        }
\item\label{Rproof.1}
Proposition~\ref{t2} can not be derived using the result that
is often called Blyth's method in the literature; cf., for example,
Theorem 7.13 in Chapter 5 of \cite{Leh98a}.
This is because the loss function $L_\kappa$ changes with the
prior $P_{\kappa,\epsilon}$ or the un-normed prior $Q_{\kappa,\epsilon}$.
But in the case where $p=1$, 
that result and our Proposition~\ref{PRD2} are derived
from essentially the same arguments.
These arguments rely on the property that 
$\phi_0$ is sufficiently close to the Bayes
procedure $\phi_\kappa$ here, in the sense that
$Q_{\kappa,\epsilon}(L_\kappa(\phi_0)) - 
	Q_{\kappa,\epsilon}(L_\kappa(\phi_\kappa))$
converges to zero as $\kappa\to 0$.
In the case where $p=2$, however, these arguments break down because
$Q_{\kappa,\epsilon}(L_\kappa(\phi_0)) - 
	Q_{\kappa,\epsilon}(L_\kappa(\phi_\kappa))$
converges to a constant. 
Nevertheless, the ideas of \cite{Bly51a}
can be adapted to also deal with this case.
This adaptation, i.e., the derivation of 
Proposition~\ref{PRD2} in the case where $p=2$,
is more intricate and comprises the bulk of the Appendix.
\item\label{Rproof.2}
Both the results of \cite{Jos69a} and our Proposition~\ref{t2} are 
derived by using ideas of \cite{Bly51a}.
But the arguments by which we derive Proposition~\ref{t2} differ from
those used by \cite{Jos69a} in two respects:
First, the weight $r_\kappa(c s/m|\lambda)$ in the
loss function in our setting depends on $s$, $\lambda$ and $\kappa$,
while this is not the case in the setting considered
by \cite{Jos69a}.
Its dependence on $s$
is also the reason why our proof establishes, for the standard procedure,
a stronger version of 
admissibility than that considered by \cite{Jos69a},
but it does not establish what is called strong admissibility by that author.
Second, at the technical core of the proof, it appears that the arguments used
by \cite{Jos69a} can not be adapted to our setting and 
that different tactics are required.  Compare  the proofs of
Lemma 6.1 and, in particular,  Lemma 6.2 in \cite{Jos69a} with 
those of lemmata~\ref{N1}--\ref{N3} in our paper.
\item \label{Rproof.3}
As a referee points out, the methods of proof used here
and by \cite{Jos69a} also rely on non-Bayesian techniques 
for proving admissibility that can be traced back to 
\cite{Bla51a} 
and that
were further developed by \cite{Bro66a} as well as \cite{Bro74a,Bro74b}.
\item\label{Rproof.4}
The function $L_\kappa(\phi)$ is not a loss function in the sense
of \cite{Leh98a}, because it depends on $r_\kappa(c s /m|\lambda)$ and hence
is not a function of the decision (e.g., the chosen confidence set)
and the parameter $\mu$ only.
Loss functions similar to $L_\kappa(\phi)$,
that depend on the decision, the true parameter and on the data,
are considered, for example, by \cite{Bro66a}, \cite{Bro74a, Bro74b} or by
\cite{Ste08a}.
\end{list}
\end{remark}

\begin{remark}\normalfont
\label{Rstrong}
It would be most interesting to know whether the standard procedure
is also strongly admissible in the sense of \cite{Jos69a}, i.e.,
whether Theorem~\ref{t1} and Theorem~\ref{t3} continue to hold
if the conditional expectations in \eqref{t1.1} and \eqref{t3.1}
are replaced by unconditional expectations. The results in this
paper do not answer this question. Indeed, Proposition~\ref{t2},
our main technical result, does not hold if the conditional expectation in
\eqref{t2.1} is replaced by an unconditional one. [Assuming otherwise,
take $\phi_1$ to be the standard confidence set in the known-variance
case to obtain a contradiction.]
It is not clear whether
our methods can be extended or adapted to also cover strong admissibility.
In the literature, there are several examples demonstrating that the
conditional behavior of tests or confidence sets can differ substantially
from their unconditional behavior; see, for example,
\cite{Bro67a,Ols73a,Rob75a,Rob79a,Bro84a};
and the references given therein.

\end{remark}

\section{On the normal-truncated-gamma prior}
\label{S4}

\subsection{General formulae}
\label{S4.1}

Throughout this section, fix $p\geq 1$.
The normal-truncated-gamma prior is a distribution on 
the parameter space $\R^p \times (0,\infty)$  that depends
on the hyper-parameters 
$\mu_\circ \in \R^p$, 
$\kappa_\circ> 0$, 
$\alpha_\circ \in \R$, 
$\beta_\circ\geq 0$, and 
$\epsilon_\circ\geq 0$.
For the density of this prior, which is defined in the following,
to be proper, we also assume that either
$\epsilon_\circ = 0$, $\alpha_\circ>0$ and $\beta_\circ>0$;
or $\epsilon_\circ>0$ and $\beta_\circ>0$; 
or $\epsilon_\circ>0$, $\alpha_\circ<0$ and $\beta_\circ=0$.
The density of this prior is given, for $\mu\in \R^p$ and $\lambda>0$, by
$$
p(\mu,\lambda) \quad=\quad
	C_{p, \alpha_\circ,\beta_\circ,\epsilon_\circ}
	\kappa_\circ^{p/2}
	\lambda^{\alpha_\circ+p/2-1} 
	e^{-\lambda (\beta_\circ + \kappa_\circ\| \mu-\mu_\circ\|^2/2)}
	\{ \lambda > \epsilon_\circ\},
$$
where the scaling constant 
$C_{p, \alpha_\circ,\beta_\circ,\epsilon_\circ}$
equals 
$(2 \pi)^{-p/2} \beta_\circ^{\alpha_\circ}/\Gamma(\alpha_\circ,
\beta_\circ \epsilon_\circ)$ in case $\beta_\circ >0$ and
$(2 \pi)^{-p/2} (-\alpha_\circ)\epsilon_\circ^{-\alpha_\circ}$ in
case $\beta_\circ=0$.
[Here $\Gamma(\alpha_\circ, \epsilon_\circ \beta_\circ)$
denotes the incomplete Gamma function $\int_{\epsilon_\circ\beta_\circ}^\infty
t^{\alpha_\circ-1} e^{-t} d t$.]
It is elementary to verify that the density $p(\cdot,\cdot)$ is proper for the
hyper-parameters as chosen here.
In the following, we use the symbol 
$\NtG(p,\mu_\circ, \kappa_\circ, \alpha_\circ,\beta_\circ, \epsilon_\circ)$
to denote the normal-truncated-gamma prior with the indicated hyper-parameters,
always assuming that these are such that the prior is proper.

\begin{remark} \normalfont
\label{rNtG}
\begin{list}{\thercnt}{
        \usecounter{rcnt}
        \setlength\itemindent{5pt}
        \setlength\leftmargin{0pt}
        \setlength\partopsep{0pt}
        }
\item
\label{rNtG.i}
In the case where $\epsilon_\circ=0$ (and hence $\alpha_0>0$ and 
$\beta_0>0$),
this prior reduces to the well-known normal-gamma prior, which is obtained
by taking $\lambda\sim \Gamma(\alpha_\circ,\beta_\circ)$ and
$\mu \|\lambda \sim N(\mu_\circ, (\kappa_\circ \lambda)^{-1} I_p)$;
cf., say, \cite{Rai61a}.
Also, in the case where case $\epsilon_\circ>0$, $\alpha_\circ>0$ and
$\beta_\circ>0$, this prior corresponds to taking 
$\lambda$ as $\Gamma(\alpha_0,\beta_0)$-distributed conditional on
the event that $\lambda > \epsilon_\circ$, and to then taking
$\mu\|\lambda \sim N(\mu_\circ, (\kappa_\circ \lambda)^{-1} I_p)$.
In general, taking a conjugate family of priors and then conditioning
on some region in parameter space, one again obtains a conjugate family.
\item
\label{rNtG.ii}
The normal-truncated-gamma prior 
$\NtG(p,0, \kappa_\circ, -p/2, 0, \epsilon_\circ)$ approximates
the non-informative prior with density $1/\lambda$
(a reference prior; cf. \citealp{Ber92a}),
as $\kappa_\circ\to 0$ and $\epsilon_\circ\to 0$,
in the sense that, for the former prior, 
the re-scaled density 
$p(\mu,\lambda)/(C_{p,-p/2,0,\epsilon_\circ}\kappa_\circ^{p/2})$
converges to $1/\lambda$ as $\kappa_\circ\to 0$ and $\epsilon_\circ\to 0$.
It is not possible to approximate this non-informative prior by
the normal-gamma priors mentioned earlier
or by those proposed by \cite{Ath86a} and \cite{Dic71a}. 
\end{list}
\end{remark}

In the next two results,
we collect some basic properties of the normal-truncated-gamma prior, that
are elementary to verify, and that are used heavily throughout
the proof of Proposition~\ref{t2}.

\begin{proposition}
\label{propNtG2}
Consider a hierarchical model with a $\NtG(p, \mu_\circ, \kappa_\circ,
\alpha_\circ, \beta_\circ, \epsilon_\circ)$ prior on the
parameters $\mu \in \R^p$ and $\lambda >0$, and 
with observations $x\in \R^p$ and $s>0$ 
so that $x$ and $s$ are independent conditional on $(\mu,\lambda)$,
and so that
$$
x\|(\mu,\lambda) \quad \sim \quad N(\mu, \lambda^{-1} I_p)
\qquad\text{and}\qquad
s \|(\mu,\lambda) \quad\sim\quad \lambda^{-1} \chi^2_m
$$
for some $m\geq 1$.
Then the posterior density of $\mu$ and $\lambda$ given $x$ and $s$ is
the density of the $\NtG(p, \mu_1,\kappa_1, \alpha_1, \beta_1,\epsilon_\circ)$
prior with
\begin{align*}
\mu_1 &\quad=\quad \frac{x + \kappa_\circ \mu_\circ}{1+\kappa_\circ}, 
\quad\quad\quad 
&\kappa_1 &\quad=\quad 1+\kappa_\circ,
\\
\alpha_1 &\quad=\quad \alpha_\circ+\frac{p+m}{2}, 
\quad\quad\quad 
&\beta_1 &\quad=\quad \beta_\circ+\frac{s}{2} + 
\frac{\kappa_\circ}{1+\kappa_\circ} \frac{\|x-\mu_\circ\|^2}{2}.
\end{align*}
\end{proposition}

\begin{proposition}
\label{propNtG1}
Under the 
$\NtG(p,\mu_\circ, \kappa_\circ,\alpha_\circ,\beta_\circ,\epsilon_\circ)$-prior,
the marginal density of $\mu$ is given by
\begin{align*}
p(\mu)\quad=\quad  
	C_{p, \alpha_\circ,\beta_\circ, \epsilon_\circ}
	\kappa_\circ^{p/2}
	\frac{
		\Gamma(\alpha_\circ+p/2, 
			\epsilon_\circ(\beta_\circ+\kappa_\circ \|
				\mu-\mu_\circ\|^2/2))
	}{
	(\beta_\circ+\kappa_\circ \|\mu-\mu_\circ\|^2/2)^{\alpha_\circ+p/2}
	}
\end{align*}
and the marginal density of $\lambda$ is given by 
\begin{align*}
p(\lambda)\quad=\quad  
C_{p,\alpha_0,\beta_0,\epsilon_0} (2 \pi)^{p/2} \lambda^{\alpha_0-1}
	e^{-\lambda \beta_0} \{ \lambda > \epsilon_0\}.
\end{align*}
Under the hierarchical model from Proposition~\ref{propNtG2},
the marginal density of $x$ and $s$ is given by
\begin{align*}
&p(x,s) \quad=\quad 
	\frac{
		C_{p, \alpha_\circ, \beta_\circ, \epsilon_\circ}}{
	2^\frac{m}{2}
	\Gamma(m/2)}
	\,
	\left(
	\frac{\kappa_\circ 
	}{1+\kappa_\circ}
	\right)^{p/2}
	\,
	s^{\frac{m}{2}-1} \,
	\frac{
		\Gamma(\alpha_1, \epsilon_\circ \beta_1)
	}{
	\beta_1^{\alpha_1}
	} ,
\end{align*}
for $\alpha_1$ and $\beta_1$ as in Proposition~\ref{propNtG2} above.
\end{proposition}

\subsection{Formulae for the  specific priors used in Section~\ref{S3}}
\label{S4.2}

Because $p_{\kappa,\epsilon}(\mu,\lambda)$ denotes the
density of the $\NtG(p,0,\kappa, -p/2,0,\epsilon)$-prior,
the following are immediate consequences of the statements
in Section~\ref{S4.1}:
Recall that $\mu_\kappa = x/(1+\kappa)$, 
and set $\beta_\kappa = (s+\frac{\kappa}{1+\kappa}\|x\|^2)/2$.
For each $\kappa>0$ and $\epsilon>0$, 
the marginal density of the observables $(x,s)$ is given by
\begin{equation}\label{marginal}\nonumber
\begin{split}
	&p_{\kappa,\epsilon}(x,s) \quad=\quad
		\frac{ 1
			}{ 
			 \Gamma(\frac{m}{2})}
		\frac{p}{2} 
		\left(
		\frac{\epsilon}{2 \pi}\;
		\frac{\kappa}{1+\kappa}
		\right)^{\frac{p}{2}}
		s^{\frac{m}{2}-1}
		\frac{
			\Gamma\left( \frac{m}{2}, 
			\epsilon \beta_\kappa\right)
		}{
		\left( 2 \beta_\kappa \right)^{\frac{m}{2}}
		}
		;
\end{split}
\end{equation}
the posterior marginal of $\lambda$ given $(x,s)$ satisfies
$$
p_{\kappa,\epsilon}(\lambda\|x,s)\quad=\quad
	\frac{\beta_\kappa^{\frac{m}{2}}}{ 
		\Gamma\left( \frac{m}{2}, \epsilon \beta_\kappa\right)}
	\lambda^{\frac{m}{2}-1}
	e^{-\lambda \beta_\kappa}
	\{ \lambda > \epsilon\};
$$
the posterior marginal of $\mu$ given $(x,s)$ satisfies
\begin{align*}
&p_{\kappa,\epsilon}(\mu\|x,s)\quad=\quad
\\
& \qquad\left( \frac{1+\kappa}{2 \pi}
\right)^\frac{p}{2}
\frac{ 
	\beta_\kappa^\frac{m}{2}
}{
	\Gamma\left( \frac{m}{2}, \epsilon \beta_\kappa
	\right)
}
\frac{ 
	\Gamma\left( \frac{m+p}{2}, \epsilon( \beta_\kappa+
	(1+\kappa)\|\mu-\mu_\kappa\|^2/2)\right)
}{
	(\beta_\kappa+(1+\kappa)\|\mu-\mu_\kappa\|^2/2)^\frac{m+p}{2}
};
\end{align*}
and
the conditional density of $\mu$ given $(x,s,\lambda)$,
i.e., $p_{\kappa,\epsilon}(\mu\|x,s,\lambda)$,
is the density the $N(\mu_\kappa, I_p/((1+\kappa)\lambda))$-distribution.
The conditional density  of $\mu$ given $x$, $s$, and $\lambda$ 
is well-defined only on the event $\lambda>\epsilon$, because that
event has probability one under the prior.
For $0<\lambda\leq \epsilon$, we define $p_{\kappa,\epsilon}(\mu\|x,s,\lambda)$
also as the density the $N(\mu_\kappa, I_p/((1+\kappa)\lambda))$-distribution,
for convenience.

For $K_{\kappa,\epsilon}$ as in Section~\ref{S3}, 
it is also easy to see that
$q_{\kappa,\epsilon}(\mu,\lambda) 
= K_{\kappa,\epsilon} p_{\kappa,\epsilon}(\mu,\lambda)$ 
and
$q_{\kappa,\epsilon}(x,s) 
= K_{\kappa,\epsilon} p_{\kappa,\epsilon}(x,s)$ 
satisfy
\begin{align*}
	q_{\kappa,\epsilon}(\mu,\lambda) &\quad=\quad
		\Gamma\left(\frac{m}{2}\right) (1+\kappa)^\frac{p}{2}
		\lambda^{-1} e^{-\lambda \kappa \frac{\|\mu\|^2}{2}}
		\;\{\lambda > \epsilon\},
\\
	q_{\kappa,\epsilon}(x,s) &\quad=\quad
		s^{\frac{m}{2}-1}
		\left(
		2 \beta_\kappa
		\right)^{-\frac{m}{2}}
		\Gamma\left( \frac{m}{2}, 
		\epsilon \beta_\kappa\right).
\end{align*}

For the case where $\kappa=0$ and $\epsilon>0$, 
note first that $\mu_0=x$ and $\beta_0 = s/2$.
Moreover, the functions $q_{0,\epsilon}(\mu,\lambda)$,
$q_{0,\epsilon}(x,s)$, as well as the conditional densities
$p_{0,\epsilon}(\lambda\|x,s)$,
$p_{0,\epsilon}(\mu\|x,s,\lambda)$, and
$p_{0,\epsilon}(\mu\|x,s)$ are well-defined by the
formulae in the preceding paragraphs (because $\beta_0 > 0$).

\section*{Acknowledgments}
We thank Larry Brown, John Hartigan,
Benedikt P\"otscher, and 
David Preinerstorfer
for inspiring discussions and suggestions.
Also, helpful feedback from the Joint Editor and two
referees is greatly appreciated.

\begin{appendix}

\section{Technical remarks for Section~\ref{S2}}
\label{AA}

\begin{remark}\normalfont \label{TR1}
Write the conditional mean in \eqref{pt3.1}
as an integral with respect to the conditional density of
$\hat{\beta}_{(\neg p})$. This density is Gaussian and can hence
be approximated pointwise from below by simple functions
that are constant in $\hat{\beta}_{(p)}$ on rectangles.
The values of these
simple functions can be chosen to be continuous functions of 
$(\beta,\sigma^2,\hat{\beta}_{(p)})$.
And $\varphi_1 (\hat{\beta}, \hat{\sigma}^2, \beta_{(p)})$ can 
be approximated pointwise from below by simple functions
in $(\hat{\beta}, \hat{\sigma}^2, \beta_{(p)})$.
Using Tonelli's theorem and the monotone convergence theorem, 
we see that
the conditional mean in \eqref{pt3.1} is approximated
pointwise
from below by finite sums of functions, where each term in such sum
is the product of a function that is measurable in
$(\hat{\beta}_{(p)}, \hat{\sigma}^2,\beta_{(p)})$
and a function that is continuous in 
$(\beta,\sigma^2,\hat{\beta}_{(p)})$.
Measurability of $\varphi_{1|\beta_{(\neg p)}}$ follows, 
because the pointwise limit of measurable functions is measurable.
\end{remark}

\begin{remark}\normalfont\label{TR2}
Recalling that $\beta_{(\neg p)}$ is fixed,
we see
for each $(\mu,\sigma^2)$
or, equivalently,
for each $(\beta_{(p)}, \sigma^2)$
that
$E_{\mu,\sigma^2}[ \phi_1(x,s,\mu,\sigma^2)] =
\E_{\beta,\sigma^2}[ 
\varphi_1(\hat{\beta},\hat{\sigma}^2,\beta_{(p)})]$,
that
$\upsilon( \phi_1(x,s,\cdot,\sigma^2)) = 
\det S_{(p)}^{-1/2}
 \upsilon( 
	\varphi_1(\hat{\beta},\hat{\sigma}^2,\cdot))$,
and that these two equalities continue to hold with
$\phi_0$ and $\varphi_0$ replacing $\phi_1$ and $\varphi_{1}$,
respectively.
\end{remark}

\begin{remark}\normalfont \label{TR3}
For the moment,
fix $\hat{\beta}_{(p)}$, $\hat{\sigma}^2$, $\beta_{(p)}$,
and $\sigma^2$ so that
the relation in 
\eqref{pt3.2}
holds for almost all
$\beta_{(\neg p)}$. 
And recall that $\varphi_{1|\beta_{(\neg p)}}(\hat{\beta}_{(p)},
\hat{\sigma}^2, \beta_{(p)}, \hat{\sigma}^2)$ is the 
conditional expectation of a function of $\hat{\beta}_{(\neg p)}$
(and $\beta_{(p)}$)
given $\hat{\beta}_{(p)}$ and $\hat{\sigma}^2$.
Write 
$\mathcal L_{\beta_{(\neg p)}}(\hat{\beta}_{(\neg p)}|
\hat{\beta}_{(p)}, \hat{\sigma}^2, \beta_{(p)},\sigma^2)$
for the corresponding (Gaussian) conditional distribution.
Parameterized by $\beta_{(\neg p)}\in \R^{d-p}$,
these distributions form an exponential family.
Because the relation in \eqref{pt3.2} holds
for almost all $\beta_{(\neg p)}$, it 
holds, in fact, for all $\beta_{(\neg p)}$; cf. 
Theorem~5.8 in Chapter 1 of \cite{Leh98a}.
The relation in the \eqref{pt3.3} now follows
from \eqref{pt3.2}, because
$\hat{\beta}_{(\neg p)}$ is a complete statistic
for the full-rank exponential family 
$\{ \mathcal L_{\beta_{(\neg p)}}(\hat{\beta}_{(\neg p)}|
\hat{\beta}_{(p)}, \hat{\sigma}^2, \beta_{(p)},\sigma^2):\;
\beta_{(\neg p)} \in \R^{d-p}\}$;
cf. Theorem~6.22 in Chapter 1 of \cite{Leh98a}.
\end{remark}

\section{Proof of Proposition~\ref{PRD1}}
\label{AB}
\label{proofrd1}

\begin{proof}[Proof of Proposition~\ref{PRD1}]
Recall that both $\phi_0$ and $\phi_\kappa$ correspond to balls
of radius $\sqrt{c s/m}$ centered at $x$ and $\mu_\kappa$, respectively.
We 
therefore have 
$\upsilon(\phi_0(x,s,\cdot,\lambda)) = 
	\upsilon(\phi_\kappa(x,s,\cdot,\lambda))$,
and the risk difference reduces to
$$
P_{\kappa,\epsilon}( \phi_0(x,s,\mu,\lambda)) \;-\; 
	P_{\kappa,\epsilon}(\phi_\kappa(x,s,\mu,\lambda)).
$$
For the standard procedure, we note, for fixed $\lambda$ and $\mu$, that
$(\|x-\mu\|^2/p)/(s/m) \| (\lambda,\mu)$ is $F$-distributed with $p$ and $m$
degrees of freedom, so that 
$P_{\kappa,\epsilon}(\phi_0(x,s,\mu,\lambda)) = F_{p,m}(c/p)$.
It remains to compute $P_{\kappa,\epsilon}(\phi_\kappa(x,s,\mu,\lambda))$. 
Because $\phi_\kappa(x,s,\mu,\lambda)$ is constant in $\lambda$, 
we have
\begin{align*}
&P_{\kappa,\epsilon}(\phi_\kappa(x,s,\mu,\lambda)) \quad=\quad
\iint\int\limits_{\|\mu-\mu_\kappa\|^2 \,<\,\, c \frac{s}{m}}
p_{\kappa,\epsilon}(\mu\|x,s) \; d \mu \; 
p_{\kappa,\epsilon}(x,s) \; d x \; d s
\\
&=\quad
\frac{p/2}{\pi^p \Gamma(m/2)}
\iint t^{\frac{m}{2}-1} 
\int\limits_{\|u\|^2\,<\,\, c t \frac{1+\kappa}{m}}
\frac{
	\Gamma(\frac{p+m}{2}, t+\|z\|^2+ \|u\|^2)
	}{
	(t+\|z\|^2+ \|u\|^2)^{\frac{p+m}{2}}
}
\; d u \; d z \; d t,
\end{align*}
where the second equality is obtained by plugging-in the formulae
for $p_{\kappa,\epsilon}(\mu\|x,s)$ and 
$p_{\kappa,\epsilon}(x,s)$ given in Section~\ref{S4.2}, 
by substituting $u$ for $(\mu-\mu_\kappa)\sqrt{(1+\kappa)(\epsilon/2)}$ in
the innermost integral,
by then substituting $z$ for $x \sqrt{(\kappa/(1+\kappa))(\epsilon/2)}$
in the middle integral, and
lastly substituting $t$ for $s (\epsilon/2)$ in the outermost integral.
Note that the resulting integral, and hence 
$P_{\kappa,\epsilon}(\phi_\kappa(x,s,\mu,\lambda))$, does not
depend on $\epsilon$.

Write $I(\tau)$ for the expression on the far right-hand side of the
preceding display with $\tau$ replacing $c$. 
Clearly $I(\tau)$ is well-defined
for each $\tau>0$.
We need to show that $I(c) = F_{p,m}((1+\kappa)c/p)$.
This will follow if we show that $I(\tau)$ is differentiable in $\tau>0$,
and that the derivatives of $I(\tau)$ and $F_{p,m}((1+\kappa)\tau/p)$ 
with respect to $\tau$ agree, i.e.,
that
\begin{equation}\label{tmp0}
\frac{\partial I(\tau)}{\partial \tau} \quad=\quad
\frac{1}{
B\left(\frac{p}{2}, \frac{m}{2}\right)}
\left(\frac{1+\kappa}{m}\right)^\frac{p}{2}
\tau^{\frac{p}{2}-1}
\left(1+\tau \frac{1+\kappa}{m}\right)^{-\frac{p+m}{2}}
\end{equation}
holds for each $\tau>0$. 

For fixed $\tau>0$ and for each $\delta>0$, the difference quotient
$(I(\tau+\delta) - I(\tau))/\delta$ can be written as
$p / (2 \pi^p \Gamma(m/2))$ multiplied by
$$
\iint t^{\frac{m}{2}-1} 
\int\limits_{  \tau t\frac{1+\kappa}{m} \,\leq\, \|u\|^2\,<\,\, 
(\tau+\delta)t\frac{1+\kappa}{m}
}
\frac{1}{\delta}
\frac{
	\Gamma(\frac{p+m}{2}, t+\|z\|^2+ \|u\|^2)
	}{
	(t+\|z\|^2+ \|u\|^2)^{\frac{p+m}{2}}
}
\; d u \; d z \; d t.
$$
Now note that the integrand in the innermost integral in the
preceding display is decreasing in $\|u\|$, and recall that
the volume of a ball of radius 
$r$ in $\R^p$ is $\pi^{p/2} r^p / \Gamma(p/2+1)$.
In view of this, the innermost integral in the preceding display 
is bounded from above by
\begin{align*}
\frac{
	\Gamma(\frac{p+m}{2}, t+\|z\|^2+ \tau t\frac{1+\kappa}{m})
	}{
	(t+\|z\|^2+ \tau t\frac{1+\kappa}{m})^{\frac{p+m}{2}}
}
\frac{1}{\delta}
\frac{\pi^{p/2}}{\Gamma(p/2+1)}
\left( 
	\left( (\tau + \delta) t\frac{1+\kappa }{m }\right)^{p/2} -
	\left( \tau t\frac{1+\kappa}{m }\right)^{p/2}
\right),
\end{align*}
and the difference quotient $(I(\tau+\delta)-I(\tau))/\delta$ is
bounded from above by
$$
\frac{1}{\pi^{\frac{p}{2}} \Gamma(\frac{p}{2})\Gamma(\frac{m}{2})}
\frac{(\tau+\delta)^\frac{p}{2} - \tau^\frac{p}{2}}{\delta}
\left(\frac{1+\kappa}{m}\right)^\frac{p}{2} 
\iint t^{\frac{p+m}{2}-1} \frac{ 
	\Gamma\left(\frac{p+m}{2}, t+\|z\|^2 + t \tau \frac{1+\kappa}{m}\right)
}{
\left(t+\|z\|^2 + t \tau \frac{1+\kappa}{m}\right)^\frac{p+m}{2}
}
\; dz \; d t.
$$
Substituting $v$ for $t+ \tau t (1+\kappa)/m$ in the integral in 
the preceding display, and using Lemma~\ref{bigint}, we
see that the upper bound is equal to
$$
\frac{1}{B\left(\frac{p}{2},\frac{m}{2}\right)}
\frac{2}{p} \frac{ (\tau+\delta)^\frac{p}{2}-\tau^\frac{p}{2}}{\delta}
\left(\frac{1+\kappa}{m}\right)^\frac{p}{2}
\left(1+\tau \frac{1+\kappa}{m}\right)^{-\frac{p+m}{2}}.
$$
Obviously, this upper bound converges to the expression on 
the right-hand side of \eqref{tmp0} as $\delta \to 0$.

In a similar fashion, the integrand in the innermost integral
in the display following \eqref{tmp0} is bounded from below  by
\begin{align*}
\frac{
	\Gamma(\frac{p+m}{2}, t+\|z\|^2+ (\tau+\delta)t\frac{1+\kappa}{m})
	}{
	(t+\|z\|^2+ (\tau+\delta) t\frac{1+\kappa}{m})^{\frac{p+m}{2}}
}
\frac{1}{\delta}
\frac{\pi^{p/2}}{\Gamma(p/2+1)}
\left( 
	\left( (\tau + \delta) t\frac{1+\kappa }{m }\right)^{p/2} -
	\left( \tau t\frac{1+\kappa}{m }\right)^{p/2}
\right).
\end{align*}
Arguing as in the preceding paragraph, we thus obtain a lower
bound for the difference quotient $I(\tau+\delta)-I(\tau))/\delta$,
which also converges to the expression on the right-hand side
of \eqref{tmp0} as $\delta\to 0$.
\end{proof}

\section{Proof of Proposition~\ref{PRD2}}
\label{AC}
\label{proofrd2}

The proof of Proposition~\ref{PRD2} is rather straight-forward in
case $p=1$ and more involved in case $p=2$. We begin with
an auxiliary result that we use for both cases. We then
prove Proposition~\ref{PRD2} for the case where $p=1$, for
completeness, and also to motivate the arguments used in the more
challenging case where $p=2$. Following this, we present
a series of lemmata that, taken together, imply the statement
in Proposition~\ref{PRD2} in case $p=2$.

\begin{lemma} \label{lemmaCompact}
Under the assumptions of Proposition~\ref{t2} and for each $\epsilon>0$,
we have
\begin{align*}
&\lim_{\kappa \to 0}
Q_{\kappa,\epsilon} \Big[ 1_C(x,s,\lambda) \;
\Big( P_{\kappa,\epsilon}(L_\kappa(\phi_1)\|x,s,\lambda) - 
	P_{\kappa,\epsilon}(L_\kappa(\phi_0)\|x,s,\lambda)
\Big) 
\Big]
\\
&=\quad
Q_{0,\epsilon} \Big[ 1_C(x,s,\lambda) \;
\Big( P_{0,\epsilon}(L_0(\phi_1)\|x,s,\lambda) - 
	P_{0,\epsilon}(L_0(\phi_0)\|x,s,\lambda)
\Big) 
\Big]
\end{align*}
for any set $C \subseteq \R^p \times (0,\infty)^2$ that is compact
in that space. Moreover, the limit in the preceding display is finite.
These statements continue to hold if the requirement 
in Proposition~\ref{t2} that $p\in \{1,2\}$
is weakened to the requirement that $p\in \N$.
\end{lemma}

\begin{proof}
Fix $p\in \N$.
The expression of the left-hand side of the preceding display is
the limit, as $\kappa \to 0$, of the integral of
$$
\Big(
P_{\kappa,\epsilon}(L_\kappa(\phi_1)\|x,s,\lambda) - 
	P_{\kappa,\epsilon}(L_\kappa(\phi_0)\|x,s,\lambda)
\Big)\;\;
p_{\kappa,\epsilon}(\lambda\|x,s)
q_{\kappa,\epsilon}(x,s) 
$$
over $(x,s,\lambda) \in C$ 
with respect to Lebesgue measure.
It suffices to show
that (i) the expression in
preceding display converges pointwise to 
the same expression with $\kappa = 0$,
and that (ii) the expression in the preceding display 
is bounded in 
absolute value, for each $(x,s,\lambda)\in C$ 
and each sufficiently small $\kappa$, 
e.g., $\kappa \leq 1$, by a function that is integrable with respect
to Lebesgue measure on $C$. 
With this, the result follows from the 
dominated convergence theorem (where the reference measure is
Lebesgue measure on $C$).

Define $\tilde{p}_{\kappa,\epsilon}(\lambda\|x,s)$ as
$p_{\kappa,\epsilon}(\lambda\|x,s)$ but with the indicator 
$\{\lambda>\epsilon\}$ replaced by $\{\lambda\geq\epsilon\}$;
cf. Section~\ref{S4.2}. Moreover, set 
$\tilde{C} = C \cap \{(x,s,\lambda): \lambda \geq \epsilon\}$.
Then integrals over $C$ with respect to $p_{\kappa,\epsilon}(\lambda\|x,s)$
coincide with integrals over $\tilde{C}$  with respect to 
$\tilde{p}_{\kappa,\epsilon}(\lambda\|x,s)$ (because 
$\tilde{p}_{\kappa,\epsilon}(\lambda\|x,s) = p_{\kappa,\epsilon}(\lambda\|x,s)$
for Lebesgue-almost all $\lambda$, and because
$\tilde{p}_{\kappa,\epsilon}(\lambda\|x,s) = 0$ whenever
$\lambda < \epsilon$).
In particular, it suffices to prove (i) and (ii) with
$\tilde{p}_{\kappa,\epsilon}(\lambda\|x,s)$  and $\tilde{C}$
replacing
$p_{\kappa,\epsilon}(\lambda\|x,s)$  and $C$, respectively.
Also, note that $\tilde{C}$ is a compact subset of $\R^p\times(0,1)^2$,
and that $\tilde{p}_{\kappa,\epsilon}(\lambda\|x,s)$  is positive and
continuous on $\tilde{C}$.

For (i), fix $(x,s,\lambda)\in \tilde{C}$.
Obviously, we have 
$\tilde{p}_{\kappa,\epsilon}(\lambda\|x,s) q_{\kappa,\epsilon}(x,s) \to
\tilde{p}_{0,\epsilon}(\lambda\|x,s) q_{0,\epsilon}(x,s)$ as $\kappa\to 0$.
The first factor in the preceding display
can be written as
\begin{align*}
&r_\kappa(c s/m|\lambda) \Big(
\upsilon(\phi_1(x,s,\cdot,\lambda)) -
\upsilon(\phi_0(x,s,\cdot,\lambda))\Big)
\\
&
-\int p_{\kappa,\epsilon}(\mu\|x,s,\lambda) \Big(
\phi_1(x,s,\mu,\lambda) -
\phi_0(x,s,\mu,\lambda)\Big)\; d \mu.
\end{align*}
Obviously, the first term in the preceding display converges
to the same term with $\kappa=0$, because
$r_\kappa(c s/m|\lambda)$ converges to 
$r_0(c s/m|\lambda)$ as
$\kappa\to 0$. For the second term, we note that
$p_{\kappa,\epsilon}(\mu\|x,s,\lambda)$ converges to
$p_{0,\epsilon}(\mu\|x,s,\lambda)$ for each $\mu$, so that
the corresponding (conditional) probability measures converge weakly.
Because $| \phi_1(x,s,\mu,\lambda) - \phi_0(x,s,\mu,\lambda)| \leq 1$,
the left-hand side of this inequality,
when viewed as a random variable with density
$p_{\kappa,\epsilon}(\mu\|x,s,\lambda)$, is uniformly integrable.
It follows that also the second expression in the preceding display
converges as required. This proves (i).

For (ii), we first note that 
$\tilde{p}_{\kappa,\epsilon}(\lambda\|x,s)$, 
as a function of 
$(x,s,\lambda)$ and $\kappa$, 
is continuous and positive on the compact set $\tilde{C}\times [0,1]$.
It follows that
$0 < \tilde{p}_\ast \leq \tilde{p}_{\kappa,\epsilon}(\lambda\|x,s)
\leq \tilde{p}^\ast < \infty$ for each $(x,s,\lambda)\in \tilde{C}$ 
and each $\kappa\in [0,1]$,
for some constants $\tilde{p}_\ast$ and $\tilde{p}^\ast$.
By a similar argument, we also have 
$0 < q_\ast \leq q_{\kappa,\epsilon}(x,s) \leq q^\ast < \infty$ and
$0 < r_\ast \leq r_\kappa(c s/m|\lambda) \leq r^\ast<\infty$
for some constants $q_\ast$, $q^\ast$, $r_\ast$, and $r^\ast$.
Moreover, recall that  
so that $q_{1,\epsilon}(x,s) = K_{1,\epsilon} p_{1,\epsilon}(x,s)$.
Our aim is to bound the product of 
$\tilde{p}_{\kappa,\epsilon}(\lambda\|x,s) q_{\kappa,\epsilon}(x,s)$ and
the expression in the preceding display, in absolute value and
for each $\kappa\in [0,1]$, by 
a function that is integrable on $\tilde{C}$.
To derive the desired bound, we first note
that $|\phi_1 - \phi_0|\leq 1$, so that the second term in the
preceding display is bounded, in absolute value, by $1$;
and the product of this  upper bound and of 
$\tilde{p}_{\kappa,\epsilon}(\lambda\|x,s) q_{\kappa,\epsilon}(x,s)$ 
is bounded by $\tilde{p}^\ast q^\ast$. Clearly, this (constant) upper bound is
integrable with respect to Lebesgue measure on $\tilde{C}$.
The product of 
$\tilde{p}_{\kappa,\epsilon}(\lambda\|x,s) q_{\kappa,\epsilon}(x,s)$ and 
the first term in the preceding display  is bounded, in absolute value,
by 
$$
\frac{\tilde{p}^\ast}{\tilde{p}_\ast}\;
\frac{q^\ast}{q_\ast}\;
K_{1,\epsilon}^{-1} \; 
\tilde{p}_{1,\epsilon}(\lambda\|x,s) \; p_{1,\epsilon}(x,s) \;
\Big(
	\upsilon( \phi_1(x,s,\cdot,\lambda)) +
	\upsilon( \phi_0(x,s,\cdot,\lambda)) 
\Big)
$$
for each $(x,s,\lambda)\in \tilde{C}$ and each $\kappa\in [0,1]$. The integral
of this upper bound with respect to Lebesgue measure on $\tilde{C}$,
and indeed also with respect to Lebesgue measure on 
$\R^p \times (0,\infty)^2$, is finite, because
\begin{align*}
& P_{1,\epsilon} \Big(\upsilon(\phi_1(x,s,\cdot,\lambda) + 
	\upsilon(\phi_0(x,s,\cdot,\lambda))\Big)
\quad\leq \quad 2 P_{1,\epsilon}
	\Big( \upsilon(\phi_0(x,s,\cdot,\lambda))\Big)
\\
&=\quad 2 \frac{ ( \pi c /m)^{p/2}}{\Gamma(p/2+1)} P_{1,\epsilon}( s^{p/2})
\quad<\quad \infty.
\end{align*}
Here, the first inequality follows because $\phi_1$ is as in Proposition~\ref{t2},
the equality holds because
$\upsilon(\phi_0(x,s,\cdot,\lambda)) = ( \pi c s/m)^{p/2} / \Gamma(p/2+1)$,
and the second inequality holds in view of Lemma~\ref{smoments}.
\end{proof}

\begin{proof}[Proof of Proposition~\ref{PRD2} in case $p=1$]
For some set $C\subseteq \R\times (0,\infty)^2$ that will be chosen
momentarily, we can write the risk difference of interest, i.e.,
$Q_{\kappa,\epsilon}(L_\kappa(\phi_1)) - 
	Q_{\kappa,\epsilon}(L_\kappa(\phi_0))$, as
\begin{align}
\label{tmp1}
&Q_{\kappa,\epsilon} \Big[ 1_C(x,s,\lambda) \;
\Big( P_{\kappa,\epsilon}(L_\kappa(\phi_1)\|x,s,\lambda) - 
	P_{\kappa,\epsilon}(L_\kappa(\phi_0)\|x,s,\lambda)
\Big) 
\Big] \;+
\\
\label{tmp2}
&Q_{\kappa,\epsilon} \Big[ 1_{C^c}(x,s,\lambda) \;
\Big( P_{\kappa,\epsilon}(L_\kappa(\phi_1)\|x,s,\lambda) - 
	P_{\kappa,\epsilon}(L_\kappa(\phi_0)\|x,s,\lambda)
\Big) 
\Big].
\end{align}
We will choose $C$ and constants $\kappa>0$ and $\epsilon>0$
so that the expression
in the preceding display is positive.

By assumption, $\phi_1(x,s,\mu,\lambda)$ and $\phi_0(x,s,\mu,\lambda)$
differ on a set of $(x,s,\mu,\lambda)$'s, i.e., on a subset of
$\R\times (0,\infty)\times \R\times(0,\infty)$,
of positive Lebesgue measure. We can choose $\epsilon>0$ so that
$\phi_1(x,s,\mu,\lambda)$ and $\phi_0(x,s,\mu,\lambda)$ also
differ on a subset of
$\R\times (0,\infty)\times \R\times(\epsilon,\infty)$,
of positive Lebesgue measure
(in view of the monotone convergence theorem).

As a preliminary consideration, we note that the
conditional risk difference 
$P_{0,\epsilon}(L_0(\phi_1)\|x,s,\lambda) - 
	P_{0,\epsilon}(L_0(\phi_0)\|x,s,\lambda)$
can be written as
\begin{equation}\label{tmp3}
\int (r_0( c s/m| \lambda) - 
	p_{0,\epsilon}(\mu\|x,s,\lambda)) 
	( \phi_1(x,s,\mu,\lambda) - \phi_0(x,s,\mu,\lambda))
\; d \mu
\end{equation}
by arguing as in the proof of Proposition~\ref{bayesprocedure}.
It is now elementary to verify that the integrand in the preceding display
is positive if $\phi_1(x,s,\mu,\lambda) \neq \phi_0(x,s,\mu,\lambda)$ and
zero otherwise.  [To this end, recall that $\phi_0(x,s,\mu,\lambda) = 1$ if 
the first factor in the integrand is negative and zero otherwise, and
that $0\leq \phi_1(x,s,\mu,\lambda) \leq 1$.]
It follows that the expression in the preceding display,
i.e., the conditional risk difference 
$P_{0,\epsilon}(L_0(\phi_1)\|x,s,\lambda) - 
	P_{0,\epsilon}(L_0(\phi_0)\|x,s,\lambda)$,
is non-negative for each $(x,s,\lambda)$.
And because $\phi_1(x,s,\mu,\lambda)$ and $\phi_0(x,s,\mu,\lambda)$ 
differ on a set of 
$(x,s,\mu,\lambda)$'s of positive Lebesgue measure,  
it is easy to see that 
$P_{0,\epsilon}(L_0(\phi_1)\|x,s,\lambda) - 
	P_{0,\epsilon}(L_0(\phi_0)\|x,s,\lambda)$,
is positive on a subset of $(x,s,\lambda)$'s 
of $\R\times (0,\infty)\times(\epsilon,\infty)$
of positive Lebesgue measure.
Noting that $p_{0,\epsilon}(\lambda\|x,s) q_{0,\epsilon}(x,s)$ is
positive on $\R\times (0,\infty)\times (\epsilon,\infty)$,
it follows that 
$$
\Delta \quad=\quad Q_{0,\epsilon} \Big[ 
 P_{0,\epsilon}(L_0(\phi_1)\|x,s,\lambda) - 
	P_{0,\epsilon}(L_0(\phi_0)\|x,s,\lambda)
\Big]\quad>\quad 0.
$$

To bound \eqref{tmp1},  we fix $\delta$ so that $0<\delta < \Delta$ and
choose a subset $C$ of $\R\times (0,\infty)^2$ so that
$$
Q_{0,\epsilon} \Big[ 1_C(x,s,\lambda)\;
\Big( P_{0,\epsilon}(L_0(\phi_1)\|x,s,\lambda) - 
	P_{0,\epsilon}(L_0(\phi_0)\|x,s,\lambda)
\Big) \Big] \quad > \quad \delta
$$
(using the considerations in the preceding paragraph and
the monotone convergence theorem).
We may also assume that $C$ is a compact
in $\R\times(0,\infty)^2$.
Now Lemma~\ref{lemmaCompact} entails that
the expression in \eqref{tmp1} is larger than $\delta/2$ 
for sufficiently small $\kappa$, i.e., $\kappa < \kappa_1$ for some
$\kappa_1>0$.

The expression in \eqref{tmp2} is bounded from below by
$Q_{\kappa,\epsilon}(L_\kappa(\phi_\kappa)) - 
	Q_{\kappa,\epsilon}(L_\kappa(\phi_0))$,
because the conditional risk 
$P_{\kappa,\epsilon}(L_\kappa(\phi)\|x,s,\lambda)$ 
is minimized
for $\phi=\phi_\kappa$
(argue as in the discussion surrounding \eqref{tmp3} but now
with $\kappa>0$).
Proposition~\ref{PRD1} entails that  the lower bound
$Q_{\kappa,\epsilon}(L_\kappa(\phi_\kappa)) - 
	Q_{\kappa,\epsilon}(L_\kappa(\phi_0))$,
and hence the expression in \eqref{tmp2}, is larger than $-\delta/2$
for sufficiently small $\kappa$, i.e., $\kappa< \kappa_2$, for some
$\kappa_2>0$.

Setting $\kappa_\ast = \min\{\kappa_1,\kappa_2\}$, we thus see
that the sum in \eqref{tmp1}-\eqref{tmp2} is positive
whenever $\kappa < \kappa_\ast$.
\end{proof}

\begin{proposition} \label{PRDa2}
Under the assumptions of Proposition~\ref{PRD2}
with $p=2$, the constant $\Delta$ defined by
$$
\Delta  \quad =\quad  Q_{0,\epsilon}( L_0(\phi_1)) - 
	Q_{0,\epsilon}(L_0(\phi_0))
$$
is well-defined, positive, and finite,
provided that $\epsilon>0$ is sufficiently small.
\end{proposition}

\begin{proof}
The proof is identical to the proof of Proposition~\ref{PRD2} for
the case $p=1$, except for the last step. In the following, when we
refer to expressions like \eqref{tmp1} or \eqref{tmp2}, etc., from that
proof, these expressions are understood to be computed for the case
considered here, i.e., for $p=2$.

For sufficiently small $\epsilon>0$, we see that
$\phi_1(x,s,\mu,\lambda)$ and $\phi_0(x,s,\mu,\lambda)$ differ
on a subset of $\R^2 \times (0,\infty)\times \R^2 \times(0,\infty)$
of positive Lebesgue measure.
Arguing as in the discussion surrounding \eqref{tmp3},
we obtain that $\Delta$ is well-defined and positive. 
It remains to show that $\Delta$ is finite.

For each $\delta<\Delta$, we can find a compact subset
$C$ of $\R^2 \times (0,\infty)^2$,
so that the expression in \eqref{tmp1} is larger than
$\delta/2$ for sufficiently small  $\kappa>0$.
And the expression in \eqref{tmp2} is bounded from below by
$Q_{\kappa,\epsilon}(L_\kappa(\phi_\kappa)) - 
	Q_{\kappa,\epsilon}(L_\kappa(\phi_0))$,
where this lower bound here converges to a finite (negative) constant
that we denote by $-\rho$; cf. Proposition~\ref{PRD1}.
Taken together, we see that the sum in \eqref{tmp1}--\eqref{tmp2}
is bounded from below by $\delta/2 - \rho$. 
On the other hand, $\phi_1$ satisfies the assumptions of  Proposition~\ref{t2}, 
so that the sum in \eqref{tmp1}--\eqref{tmp2} is non-positive;
cf. \eqref{key}.
It follows that $0 \geq \delta/2 - \rho$, i.e., $\delta \leq 2 \rho$. 
Since this holds for
each $\delta<\Delta$, we get that $\Delta \leq 2 \rho$, whence $\Delta$
is finite as claimed.
\end{proof}

Throughout the following, assume that the assumptions
of Proposition~\ref{PRD2} are satisfied and that $p=2$,
and fix $\epsilon>0$ so that Proposition~\ref{PRDa2} applies.
In particular, the constant $\Delta$ defined in that proposition
is a positive real number.
For each $\kappa\geq 0$,
decompose 
$Q_{\kappa,\epsilon}(L_\kappa(\phi_1)) - 
	Q_{\kappa,\epsilon}(L_\kappa(\phi_0))$ as
$$
Q_{\kappa,\epsilon}(L_\kappa(\phi_1)) - 
	Q_{\kappa,\epsilon}(L_\kappa(\phi_0))
\quad=\quad
M_{\kappa,\epsilon} -
N^{(1)}_{\kappa,\epsilon} -
N^{(2)}_{\kappa,\epsilon} -
N^{(3)}_{\kappa,\epsilon},
$$
where $M_{\kappa,\epsilon}$ is as in \eqref{tmp1} (with $p=2$) for some
compact subset $C$ of $\R^2\times (0,\infty)^2$, and where
the sum $-N^{(1)}_{\kappa,\epsilon} - 
	N^{(2)}_{\kappa,\epsilon} - N^{(3)}_{\kappa,\epsilon}$
further decomposes the expression in \eqref{tmp2} (with $p=2$) as follows:
For each $(x,s,\lambda)$, 
write the conditional risk difference in \eqref{tmp2} as the expression
in \eqref{tmp3} with $\kappa$ replacing $0$, decompose the range of $\mu$,
i.e., $\R^2$, into three disjoints sets $A^{(i)}(x,s)$ ($i=1,2,3$) 
for each $(x,s)$, and
set
\begin{align*}
&N^{(i)}_{\kappa,\epsilon}\quad=\quad
\\&
\iiint\limits_{C^c} 
\int\limits_{A^{(i)}(x,s)} \Big(r_\kappa(c s/m|\lambda) -
	p_{\kappa,\epsilon}( \mu\|x,s,\lambda)\Big) 
	\Big(\phi_0(x,s,\mu,\lambda) - \phi_1(x,s,\mu,\lambda)\Big) 
	\; d\mu\; 
\\&
\qquad\qquad\qquad 
	p_{\kappa,\epsilon}(\lambda\|x,s) \; d \lambda \;
	q_{\kappa,\epsilon}(x,s)\; d x\; d s
\end{align*}
for $i=1,2,3$,
where
\begin{align*}
A^{(1)}(x,s) & \quad = \quad \{\mu:\; \|\mu - x\|^2 < c s/m\},\\
A^{(2)}(x,s) & \quad = \quad \{\mu:\; c s/m \leq \|\mu- x\|^2 < (c+1)s/m\}
	\text{, and}\\
A^{(3)}(x,s) & \quad = \quad \{\mu:\; (c+1) s/m \leq \|\mu- x\|^2 \}.
\end{align*}
Note that we have $M_{0,\epsilon}\geq 0$ and 
$N_{0,\epsilon}^{(i)}\leq 0$ for $i=1,2,3$;
cf. the discussion surrounding \eqref{tmp3}.
For later use, we also note that $M_{0,\epsilon}$ and 
$-N_{0,\epsilon}^{(i)}$ for $i=1,2,3$
are non-decreasing in $C$
(in the sense that, say, $M_{0,\epsilon}$, 
does not decrease if $C$ is replaced
by a superset $\tilde{C}$ of $C$). 
Moreover, $M_{0,\epsilon}$ approaches
the constant
$\Delta$ from Proposition~\ref{PRDa2}
and the $N_{0,\epsilon}^{(i)}$'s 
approach $0$ from below as $C$ increases.
In other words, for fixed $\rho>0$, we have
$\Delta-\rho < M_{0,\epsilon} \leq \Delta$ and
$-\rho < N_{0,\epsilon}^{(i)} \leq 0$ for $i=1,2,3$,
provided only that $C\subseteq \R^2 \times (0,\infty)^2$ is sufficiently
large;
and without loss of generality, we may always assume that
$C$ is compact in that space.
The next four results show that we can choose the set $C$
and the constant $\kappa>0$, so that
$M_{\kappa,\epsilon} > \Delta/2$ and so
that $N^{(i)}_{\kappa,\epsilon} < \Delta/6$ for $i=1,2,3$, and thus
prove Proposition~\ref{PRD2} in the case where $p=2$.
In the following, when we say that the constant $\kappa$ is 
sufficiently small, we mean that $\kappa<\kappa^\ast$ for
some finite number $\kappa^\ast>0$. Similarly, when we say that a set $C$
is sufficiently large, we mean that $C^\ast \subseteq C$
for some bounded set $C^\ast\neq \emptyset$.

\begin{lemma}
\label{M}
Under the assumptions of Proposition~\ref{PRD2} with $p=2$,
we have $M_{\kappa,\epsilon} > \Delta/2$ provided only that  the set
$C\subseteq \R^2\times (0,\infty)^2$ is sufficiently
large and compact, and that $\kappa$ is sufficiently small
and positive (where $\Delta$ and $\epsilon$ are as 
in Proposition~\ref{PRDa2}).
\end{lemma}

\begin{proof}
The result is derived by arguing as in the paragraph following
\eqref{tmp3} but now with $p=2$, mutatis mutandis.
\end{proof}

\begin{lemma}
\label{N1}
Under the assumptions of Proposition~\ref{PRD2} with $p=2$,
we have 
$N_{\kappa,\epsilon}^{(1)} < \Delta/6$ provided only that the set
$C\subseteq \R^2\times (0,\infty)^2$ is sufficiently large,    
and that $\kappa$ is sufficiently small and positive
(where $\Delta$ and $\epsilon$ are as in Proposition~\ref{PRDa2}).
\end{lemma}

\begin{proof}
We first derive a convenient upper bound for $N_{0,\epsilon}^{(1)}$.
Noting that $\phi_0(x,s,\mu,\lambda) = 1$ whenever 
$\mu \in A^{(1)}(x,s)$, we
can write $N_{0,\epsilon}^{(1)}$ as the integral over  $C^c$ of
\begin{align*}
q_{0,\epsilon}(x,s)\; p_{0,\epsilon}(\lambda\|x,s)\
	\Bigg( & r_0(c s/m|\lambda) \upsilon_1(x,s,\lambda) 
	\\
 	&\;\;\;- 
	\int\limits_{A^{(1)}(x,s)}
		p_{0,\epsilon}(\mu\|x,s,\lambda) 
		(1 - \phi_1(x,s,\mu,\lambda)) \; d \mu \Bigg)
\end{align*}
with respect to $x$, $s$ and $\lambda$,
where $\upsilon_1(x,s,\lambda) = 
\int_{A^{(1)}(x,s)} (1-\phi_1(x,s,\mu,\lambda)) \; d \mu$.
For each $(x,s,\lambda)$, choose $c_1 = c_1(x,s,\lambda)$ so that
$\pi c s/m  - \pi c_1 s/m = \upsilon_1(x,s,\lambda)$, and note
that $0 \leq c_1 \leq c$. 
We obtain that
the expression in the preceding display is bounded from above by
\begin{align*}
q_{0,\epsilon}(x,s)\; p_{0,\epsilon}(\lambda\|x,s)\
	\Bigg( & r_0(c s/m|\lambda) \upsilon_1(x,s,\lambda) 
	\\
 	&\;\;\;- 
	\int\limits_{c_1 s/m \leq \|\mu-x\|^2 \leq c s/m}
		p_{0,\epsilon}(\mu\|x,s,\lambda) 
		 \; d \mu \Bigg)
\end{align*}
(by recalling  that $p_0(\mu\|x,s,\lambda)$ is radially
symmetric in $\mu$ around $x$ and decreasing in $\|\mu-x\|$,  by picturing
the set $A^{(1)}(x,s)$ as a subset of the plane,
and by a little reflection).
Using the results in Section~\ref{S4.2}, we see that
$$
q_{0,\epsilon}(x,s) \; p_{0,\epsilon}(\lambda\|x,s) \;
p_{0,\epsilon}(\mu\|x,s,\lambda) \quad=\quad
\frac{s^{\frac{m}{2}-1} \lambda^\frac{m}{2}}{\pi 2^{\frac{m}{2}+1}}
e^{-\frac{\lambda}{2}( s + \|\mu-x\|^2)} \{\lambda > \epsilon\}.
$$
Recalling that $r_0(c s/m|\lambda)$ equals
$p_{0,\epsilon}(\mu\|x,s,\lambda)$ evaluated at $\|\mu-x\|^2 = c s/m$,
we can write the upper bound in the second-to-last display  as
$$
\frac{s^{\frac{m}{2}-1}\lambda^\frac{m}{2}}{ \pi 2^{\frac{m}{2}+1}}
\int\limits_{ c_1 s/m \leq \|\mu-x\|^2\leq c s/m}
\left(
e^{-\frac{\lambda}{2}(s + c s/m) }
-
e^{-\frac{\lambda}{2}(s + \|\mu-x\|^2) }
\right) \; d \mu \;\{\lambda>\epsilon\}.
$$
Because the exponential function is convex, we obtain that
the upper bound in the preceding display is bounded from above by
\begin{align*}
\frac{s^{\frac{m}{2}-1} \lambda^{\frac{m}{2}+1}}{\pi 2^{\frac{m}{2}+2}}
e^{-\frac{\lambda}{2}( s + c s/m)}
\int\limits_{ c_1 s/m \leq \|\mu-x\|^2\leq c s/m}
\left( \|\mu-x\|^2 - c s/m \right) 
\; d \mu \;\{\lambda>\epsilon\}.
\end{align*}
Evaluating the integral in the preceding display  
using standard methods, we obtain that
$$
N_{0,\epsilon}^{(1)} \quad\leq \quad
-\iiint \limits_{C^c}
\frac{s^{\frac{m}{2}-1} \lambda^{\frac{m}{2}+1}}{ \pi^2 2^{\frac{m}{2}+3}}
\;
e^{-\frac{\lambda}{2}(s + c s/m)} \;
\upsilon_1^2(x,s,\lambda)
\;\{\lambda>\epsilon\}\;
d \lambda \; d x \; d s.
$$
Note that this upper bound is non-positive.
Recalling that $N_{0,\epsilon}^{(1)}$ can be made arbitrarily close to
zero by choosing $C$ sufficiently large, we obtain, for each
$\delta>0$, that
\begin{equation}\label{tmp5}
\iiint \limits_{C^c}
\frac{s^{\frac{m}{2}-1} \lambda^{\frac{m}{2}+1}}{ \pi^2 2^{\frac{m}{2}+3}}
\;
e^{-\frac{\lambda}{2}(s + c s/m)} \;
\upsilon_1^2(x,s,\lambda)
\;\{\lambda>\epsilon\}\;
d \lambda \; d x \; d s
\quad< \quad \delta
\end{equation}
provided only that  $C$ is sufficiently large.

In the next step, we derive an upper bound 
for $N_{\kappa,\epsilon}^{(1)}$ for
$\kappa>0$, by arguments similar to those used in the preceding paragraph.
Let $c_2=c_2(x,s,\lambda,\kappa)$ and $c_3 = c_3(x,s,\kappa)$ be such that
$\sqrt{c_3 s/m} = \sqrt{c s/m} + \|x\| \kappa/(1+\kappa)$ and such that
$\pi c_3 s/m - \pi c_2 s/m = \upsilon_1(x,s,\lambda)$. 
The quantity $N_{\kappa,\epsilon}^{(1)}$
is bounded from above by the integral over $C^c$ of
\begin{align*}
q_{\kappa,\epsilon}(x,s)\; p_{\kappa,\epsilon}(\lambda\|x,s)\
	\Bigg( & r_\kappa(c s/m|\lambda) 
	\upsilon_1(x,s,\lambda) 
	\\
 	&\;\;\;- 
	\int\limits_{c_2 s/m \leq \|\mu-\mu_\kappa\|^2 \leq c_3 s/m}
		p_{\kappa,\epsilon}(\mu\|x,s,\lambda) 
		 \; d \mu \Bigg),
\end{align*}
because $p_{\kappa,\epsilon}(\mu\|x,s)$ is radially symmetric in $\mu$
around $\mu_\kappa$ and decreasing in $\|\mu-\mu_\kappa\|$.
Again using the formulas in Section~\ref{S4.2},
we can write the expression in the 
preceding display as
\begin{align*}
&(1+\kappa)
\frac{s^{\frac{m}{2}-1} \lambda^{\frac{m}{2}}}{\pi 2^{\frac{m}{2}+1}}
\;\times
\\&\qquad
\int\limits_{ \frac{c_2 s}{m} \leq \|\mu-\mu_\kappa\|^2\leq\frac{c_3 s}{m}}
\left(
e^{-\frac{\lambda}{2}( 2 \beta_\kappa+(1+\kappa) \frac{c s}{m})}
-
e^{-\frac{\lambda}{2}( 2 \beta_\kappa+(1+\kappa) \|\mu-\mu_\kappa\|^2)}
\right)
\; d \mu \;\{\lambda>\epsilon\}.
\end{align*}
In the preceding display, the integrand 
is increasing in $\|\mu-\mu_\kappa\|^2$, so that the integral
is bounded from above by
\begin{align*}
&\int\limits_{\frac{c_2 s}{m} \leq \|\mu-\mu_\kappa\|^2\leq \frac{c_3 s}{m}}
\left(
e^{-\frac{\lambda}{2}( 2 \beta_\kappa+(1+\kappa) \frac{c s}{m})}
-
e^{-\frac{\lambda}{2}( 2 \beta_\kappa+(1+\kappa) \frac{c_3 s}{m})}
\right)\; d \mu
\\
&\leq\quad 
(1+\kappa)\frac{\lambda}{2}
e^{-\frac{\lambda}{2}( 2 \beta_\kappa+(1+\kappa) \frac{c s}{m})}
\left( \frac{c_3 s}{m} - \frac{ c s}{m}\right)
\;\upsilon_1(x,s,\lambda),
\end{align*}
where the inequality is obtained by using convexity of the exponential
function.
By the arguments presented so far, we see that
\begin{equation}\label{tmp6}
\begin{split}
N_{\kappa,\epsilon}^{(1)}\quad\leq \quad
(1+\kappa)^2 \iiint\limits_{C^c}
&
\frac{s^{\frac{m}{2}-1} \lambda^{\frac{m}{2}+1}}{\pi 2^{\frac{m}{2}+2}}
\;
e^{-\frac{\lambda}{2}(2 \beta_\kappa+(1+\kappa)\frac{c s}{m}) }\;
\\
&\;\;
\left( \frac{c_3 s}{m} - \frac{ c s}{m}\right) \;
\upsilon_1(x,s,\lambda) \; \{\lambda>\epsilon\}
\; d \lambda \; d x \; d s
\end{split}
\end{equation}
for each $\kappa>0$.

We now combine
the upper bound on $N_{\kappa,\epsilon}^{(1)}$
given in \eqref{tmp6}
with 
\eqref{tmp5}. To this end, we use H\"older's inequality
(where the reference measure 
has the density 
$s^{\frac{m}{2}-1} \lambda^{\frac{m}{2}+1} \exp( -\frac{\lambda}{2}
(2\beta_\kappa + (1+\kappa) \frac{c s}{m})$
on $C^c$)
and the fact that the exponential function is monotone
to conclude that $N_{\kappa,\epsilon}^{(1)}$ is bounded from above by
\begin{equation} \nonumber
\begin{split}
(1+\kappa)^2 
\left( \iiint \limits_{C^c}
\frac{s^{\frac{m}{2}-1}\lambda^{\frac{m}{2}+1}}{\pi 2^{\frac{m}{2}+2}}
e^{-\frac{\lambda}{2} (s + \frac{c s}{m})}
\upsilon_1^2(x,s,\lambda)
\;\{\lambda>\epsilon\}\; d \lambda \; d x \; d s
\right)^{1/2}
 \\ 
\quad 
\left( \iiint\limits_{C^c}
\frac{s^{\frac{m}{2}-1}\lambda^{\frac{m}{2}+1}}{\pi 2^{\frac{m}{2}+2}}
e^{- \lambda \beta_\kappa}
\left(\frac{c_3 s}{m} - \frac{ c s}{m}
\right)^2 \; \{\lambda>\epsilon\}\; d \lambda \; d x \; d s
 \right)^{1/2}
\end{split}
\end{equation}
for each $\kappa>0$.
This upper bound is the product of three factors. The first one
is smaller than, say, $2$ provided that $\kappa$ is sufficiently small.
In view of \eqref{tmp5} and
for fixed $\delta>0$, 
the second one is smaller than $(2 \pi \delta)^{1/2}$
provided that $C$ is sufficiently large.
To bound the third factor, we extend the integral over
the whole space $\R^2\times(0,\infty)^2$ and note that,
in the resulting upper bound, the innermost integral, i.e.,
$\int_{\epsilon}^\infty \lambda^{m/2+1}e^{-\lambda \beta_\kappa}\;d\lambda$,
equals $\Gamma(m/2+2, \epsilon\beta_\kappa)/\beta_\kappa^{m/2+2}$.
Moreover, setting $\rho = \kappa/(1+\kappa)$ and noting that
$0<\rho<1$, we have
$(c_3 s/m - c s/m)^2 = (2 \sqrt{ c s/m} \|x\|\rho + 
			\|x\|^2\rho^2)^2 
		\leq \rho ( 2 \sqrt{ c s/m } \|x\| \sqrt{\rho}
				+ \|x\|^2 \rho)^2$.
Using this inequality to further bound the resulting upper bound,
substituting $y$ for $x \sqrt{\rho}$,  and simplifying, we see that
the third factor in the preceding display is bounded from above by
the square root of
$$
\frac{1}{\pi}
\int\limits_{\R^2}
\int_0^\infty
	s^{\frac{m}{2}-1}
	\frac{\Gamma(\frac{m}{2}+2, \frac{\epsilon}{2} (s+\|y\|^2))}{
		(s+\|y\|^2)^{\frac{m}{2}+2}}
	\left(
		2 \sqrt{\frac{c s}{m}} \|y\| + \|y\|^2
	\right)^2
\; d s
\; d y.
$$
The expression in the preceding display is bounded by a finite constant
that we denote by $\Lambda^2$, in view of Lemma~\ref{bigint}.
[To apply the lemma, substitute $t$ for $s \epsilon/2$,
substitute $z$ for $y \sqrt{\epsilon/2}$, and expand the square
in the integrand. This results in a sum of three integrals,
where each is finite by Lemma~\ref{bigint}.]
Taken together, we see that  $N_{\kappa,\epsilon}^{(1)}$ is 
bounded from above by
$2 ( 2 \pi \delta)^{1/2}\Lambda$. The proof is completed by appropriate
choice of $\delta$.
\end{proof}

\begin{lemma}
\label{N2}
Under the assumptions of Proposition~\ref{PRD2} with $p=2$,
we have $N_{\kappa,\epsilon}^{(2)} < \Delta/6$ provided only that
$C\subseteq \R^2\times (0,\infty)^2$ is sufficiently large,  
and that $\kappa$ is sufficiently small and positive
(where $\Delta$ and $\epsilon$ are as in Proposition~\ref{PRDa2}).
\end{lemma}

\begin{proof}
The proof relies on ideas similar to those used
earlier in the proof of Lemma~\ref{N1}, and on some
additional considerations to deal with issues that do not
occur in the preceding proof.
Again, we first obtain an upper bound for $N_{0,\epsilon}^{(2)}$:
Since $\phi_{0,\epsilon}(x,s,\mu,\lambda)=0$ 
for $\mu \in A^{(2)}(x,s)$, we can write
$N_{0,\epsilon}^{(2)}$ as the integral over $C^c$ of
\begin{align*}
q_{0,\epsilon}(x,s) \;p_{0,\epsilon}(\lambda\|x,s) \;\Bigg(
\int\limits_{A^{(2)}(x,s)} 
&
	p_{0,\epsilon}(\mu\|x,s,\lambda) \phi_1(x,s,\mu,\lambda)\; d \mu
\\
&
\quad
\;\;-\;\; 
r_0( c s/m|\lambda) \upsilon_2(x,s,\lambda)
\Bigg),
\end{align*}
where 
$\upsilon_2(x,s,\lambda) = 
	\int_{A^{(2)}(x,s)} \phi_1(x,s,\mu,\lambda) \; d \mu$.
Now choose $c_1 = c_1(x,s,\lambda)$ so that
$\pi c_1 s/m - \pi c s/m = \upsilon_2(x,s,\lambda)$, and note that
$c\leq c_1 \leq c+1$.  
With this, the expression in the preceding display is bounded by
the following sequence of expressions.
\begin{align*}
& 
	q_{0,\epsilon}(x,s) \;p_{0,\epsilon}(\lambda\|x,s)\;
	\Bigg(
	\int\limits_{ \frac{c s}{m} \leq \|\mu - x\|^2\leq\frac{c_1 s}{m}}
	p_{0,\epsilon}(\mu\|x,s,\lambda) \; d \mu
\\
&\qquad\qquad\qquad\qquad\qquad\qquad\qquad
	\;\;-\;\; r_0( c s/m|\lambda) \upsilon_2(x,s,\lambda)
	\Bigg)
\\
&= \quad 
	\frac{s^{\frac{m}{2}-1} \lambda^\frac{m}{2}}{
		\pi 2^{\frac{m}{2}+1}}
	\int\limits_{ \frac{c s}{m} \leq \|\mu - x\|^2 \leq\frac{c_1 s}{m}}
	\left(
	e^{-\frac{\lambda}{2}(s+\|\mu-x\|^2)}
	-
	e^{-\frac{\lambda}{2}(s+\frac{c s}{m})}
	\right)
	\; d \mu \;\{\lambda>\epsilon\}
\\
&\leq \quad
	\frac{s^{\frac{m}{2}-1} \lambda^{\frac{m}{2}+1}}{
		\pi 2^{\frac{m}{2}+2}}
	e^{-\frac{\lambda}{2}(s+\frac{c_1 s}{m})}
	\int\limits_{ \frac{c s}{m} \leq \|\mu - x\|^2 \leq\frac{c_1 s}{m}}
	\left(
	\frac{c s}{m} - \|\mu-x\|^2 
	\right)
	\; d \mu\;\{\lambda>\epsilon\}
\\
&= \quad
	- \frac{s^{\frac{m}{2}-1} \lambda^{\frac{m}{2}+1}}{
		\pi^2 2^{\frac{m}{2}+3}}
	e^{-\frac{\lambda}{2}(s+\frac{c_1 s}{m})}
	\upsilon_2^2(x,s,\lambda)\;\{\lambda>\epsilon\}
	.
\end{align*}
In this sequence of expressions, the
first one is an upper bound of the expression in the second-to-last display,
because $p_{0,\epsilon}(\mu\|x,s,\lambda)$ 
is radially symmetric around $x$ and monotone in $\|\mu-x\|$;
the first equality is derived by plugging-in the
formulas from Section~\ref{S4.2};
the first inequality is derived by first using the convexity
and then the monotonicity of the exponential function;
and the last equality is obtained by elementary integration.
For fixed $\delta>0$, recall that we have 
$-\delta < N_{0,\epsilon}^{(2)} \leq 0$
provided that $C$ is sufficiently large.
Because $N_{0,\epsilon}^{(2)}$ is bounded from 
above by the integral over $C^c$
of the expression at the far right-hand side of the preceding display,
it follows that
\begin{equation}\label{tmp7}
	\iiint\limits_{C^c}
	 \frac{s^{\frac{m}{2}-1} \lambda^{\frac{m}{2}+1}}{
		\pi^2 2^{\frac{m}{2}+3}}
	\;
	e^{-\frac{\lambda}{2}(s+\frac{c_1 s}{m})}
	\;
	\upsilon_2^2(x,s,\lambda)
	\;
	\{\lambda>\epsilon\}
	\; d \lambda \; d x \; d s
	\quad<\quad \delta,
\end{equation}
provided only that $C$ is sufficiently large.

To bound $N_{\kappa,\epsilon}^{(2)}$ from above,
let $c_2  = c_2(x,s,\kappa)$ and $c_3 = c_3(x,s,\lambda,\kappa)$ be so that
$\sqrt{c_2 s/m} = \max\{ \sqrt{c s/m} - \|x\|\kappa/(1+\kappa), 0\}$
and so that $\pi c_3 s/m - \pi c_2 s/m = \upsilon_2(x,s,\lambda)$.
With this, $N_{\kappa,\epsilon}^{(2)}$ is bounded from above 
by the integral
over $C^c$ of
\begin{align}
\nonumber
&q_{\kappa,\epsilon}(x,s) \; p_{\kappa,\epsilon}(\lambda\|x,s)\;\Bigg(
  \int\limits_{ \frac{c_2 s}{m} 
  \leq \|\mu-\mu_\kappa\|^2 \leq \frac{ c_3 s}{m}}
  p_{\kappa,\epsilon}(\mu\|x,s,\lambda) \; d \mu 
\\
\nonumber
&\qquad\qquad\qquad\qquad\qquad\qquad\qquad\qquad\qquad
- r_\kappa\left(\left.\frac{c s}{m}\right|\lambda\right)
  \upsilon_2(x,s,\lambda)
  \Bigg)
\\
\
\label{N2.3}
&=\quad
 \frac{1+\kappa
}{
	\pi 2^{\frac{m}{2}+1}}
s^{\frac{m}{2}-1 } \lambda^{\frac{m}{2}}
\int\limits_{ \frac{c_2 s}{m} \leq \|\mu-\mu_\kappa\|^2\leq\frac{c_3 s}{m}}
\Big(
e^{-\frac{\lambda}{2}( 2 \beta_\kappa+(1+\kappa)\|\mu-\mu_\kappa\|^2)}
\\
\nonumber
&\qquad\qquad\qquad\qquad\qquad\qquad\qquad\qquad\qquad
-
e^{-\frac{\lambda}{2}( 2 \beta_\kappa+(1+\kappa)\frac{c s}{m})}
\Big)
\; d \mu \; \{\lambda > \epsilon\}
\\
\nonumber
&\leq\quad
 \frac{1+\kappa
}{
	\pi 2^{\frac{m}{2}+1}}
s^{\frac{m}{2}-1 } \lambda^{\frac{m}{2}}
\left(e^{-\frac{\lambda}{2}( 2 \beta_\kappa+(1+\kappa)\frac{c_2 s}{m})}
-
e^{-\frac{\lambda}{2}( 2 \beta_\kappa+(1+\kappa)\frac{c s}{m})}
\right)
\upsilon_2(x,s,\lambda)
\; \{\lambda > \epsilon\}
\\
\label{N2.2}
&\leq\quad
 \frac{(1+\kappa)^2
}{
	\pi 2^{\frac{m}{2}+2}}
s^{\frac{m}{2}-1 } \lambda^{\frac{m}{2}+1}
e^{-\frac{\lambda}{2}( 2 \beta_\kappa+(1+\kappa)\frac{c_2 s}{m})}
\left(
\frac{c s}{m} - \frac{c_2 s}{m}
\right)
\upsilon_2(x,s,\lambda)
\; \{\lambda > \epsilon\}.
\end{align}
Here, the upper bound follows 
because $p_{\kappa,\epsilon}(\mu\|x,s,\lambda)$ is radially
symmetric around $\mu_\kappa$ and monotone in $\|\mu-\mu_\kappa\|$;
the equality is obtained by plugging-in the formulae
from Section~\ref{S4.2} and simplifying;
and the two inequalities are obtained by using the
monotonicity and the convexity of the exponential function,
respectively.

Up to this point
the proof has proceeded  similarly to the proof of
Lemma~\ref{N1}. But now we find that the inequality in \eqref{tmp7}
and the upper bound for $N_{\kappa,\epsilon}^{(2)}$ 
that can be obtained from the preceding paragraph can not
be combined as in the proof of Lemma~\ref{N1}.
A more detailed analysis appears to be in order.
To this end, we  decompose the range of $(x,s,\lambda)$, i.e., 
$\R^2 \times (0,\infty)^2$, as
$$
D_\kappa \;\;\cup \;\;
\left(D_\kappa^c\cap E_{l}\right) \;\;\cup \;\;
\left(D_\kappa^c \cap E_{l}^c \right) 
$$
for 
$D_\kappa = \{(x,s,\lambda): 
	\|x\| \kappa/(1+\kappa) \geq \sqrt{ c s/m}\}$, and
$E_{l} = \{(x,s,\lambda):  
	s \lambda  \leq l \}$, where $l$ will be chosen later.
This corresponds to the decomposition
$N_{\kappa,\epsilon}^{(2)} =  
N_{\kappa,\epsilon}^{(2,1)} + 
N_{\kappa,\epsilon}^{(2,2)} + 
N_{\kappa,\epsilon}^{(2,3)}$,
where $N_{\kappa,\epsilon}^{(2,i)}$ is defined as 
$N_{\kappa,\epsilon}^{(2)}$ but
with the set $C^c$ replaced by the
intersection of $C^c$ and the $i$-th set in the union in the
preceding display, $i=1,2,3$.

To bound $N_{\kappa,\epsilon}^{(2,1)}$ we first note that
$N_{\kappa,\epsilon}^{(2,1)}$ is bounded by the integral
of \eqref{N2.3} over $C^c \cap D_\kappa$ or over the larger
set $D_\kappa$.
Moreover, the integral in \eqref{N2.3} is bounded by
$\exp(-\lambda\beta_\kappa) \upsilon_2(x,s,\lambda)$,
because the exponential function is positive and
monotone. Since $\upsilon_2(x,s,\lambda)$ is bounded by the measure of
$A^{(2)}(x,s)$, i.e, $\upsilon_2(x,s,\lambda) \leq \pi s/m$, it is
straight-forward to see that
$N_{\kappa,\epsilon}^{(2,1)}$ is bounded by 
\begin{align*}
&
\frac{1+\kappa}{m 2^{\frac{m}{2}+1}}
\iint\limits_{ \|x\|\frac{\kappa}{1+\kappa}  \geq \
\sqrt{ \frac{c s}{m}}}
s^\frac{m}{2}
\int_\epsilon^\infty
\lambda^\frac{m}{2} e^{-\lambda \beta_\kappa}
\; d\lambda \; d x \; d s
\\
&=\quad
\frac{1+\kappa}{m}
\iint\limits_{ \|x\|\frac{\kappa}{1+\kappa}  \geq \
\sqrt{ \frac{c s}{m}}}
s^\frac{m}{2}
\frac{ \Gamma\left(\frac{m}{2}+1, \epsilon \beta_\kappa\right)}{
	(2 \beta_\kappa)^{\frac{m}{2}+1}}
\;d x \; d s
\\
&=\quad
\frac{ 2 (1+\kappa)^2}{m \epsilon\kappa}
\iint \limits_{\|z\|^2 \frac{m}{c} \frac{\kappa}{1+\kappa} > t}
t^\frac{m}{2} \frac{ \Gamma\left(\frac{m}{2}+1, t+\|z\|^2\right)}{
	(t+\|z\|^2)^{\frac{m}{2}+1}}
\; d z \; d t,
\end{align*}
where the last equality is obtained by substituting
$t$ for $s \epsilon/2$ and
by substituting $z$ for $x \sqrt{ \kappa\epsilon/( 2 (1+\kappa)}$.
Using Lemma~\ref{lemmaD}, it is now easy to see
that the upper bound in the 
preceding display converges to zero as $\kappa\to 0$.
In particular, we see for each $\delta>0$ that 
\begin{equation}\label{boundN21}
N_{\kappa,\epsilon}^{(2,1)}  \quad <\quad\delta
\qquad \qquad \text{if $\kappa < \kappa_1(\delta,\epsilon,m,c)$},
\end{equation}
where $\kappa_1(\delta,\epsilon,m,c)$ is an appropriate positive constant
depending only on the indicated quantities.

Next, $N_{\kappa,\epsilon}^{(2,2)}$
is bounded by the integral of \eqref{N2.2} over
$C^c \cap D_\kappa^c\cap E_{l}$. Using H{\"o}lder's inequality,
we see that
$N_{\kappa,\epsilon}^{(2,2)}$ is bounded by
\begin{align*}
\frac{(1+\kappa)^2}{\pi 2^{\frac{m}{2}+2}}
&
\left(\;\;
\iiint \limits_{C^c \cap D^c_\kappa \cap E_{l}}
s^{\frac{m}{2}-1} \lambda^{\frac{m}{2}+1}
e^{-\frac{\lambda}{2}( 2\beta_\kappa + (1+\kappa) \frac{c_2 s}{m})}
\upsilon_2^2(x,s,\lambda)
\;\{\lambda>\epsilon\}
\; d \lambda \; d x \; d s
\right)^\frac{1}{2}
\\
&
\left(\;\;
\iiint \limits_{C^c \cap D^c_\kappa \cap E_{l}}
s^{\frac{m}{2}-1} \lambda^{\frac{m}{2}+1}
e^{-\frac{\lambda}{2}( 2\beta_\kappa + (1+\kappa) \frac{c_2 s}{m})}
\left(\frac{c s}{m} - \frac{c_2 s}{m}\right)^2
\;\{\lambda>\epsilon\}
\; d \lambda \; d x \; d s
\right)^\frac{1}{2}
.
\end{align*}
The expression in the preceding display is the product of three factors.
The first factor is smaller than, say, $2$ if
$\kappa$ is sufficiently small, e.g., $\kappa < \kappa_2$,
for some $\kappa_2 = \kappa_2(m)$.
Concerning the  second factor, we first note that,
for $(x,s,\lambda) \in E_{l}$, the exponential function in 
the integrand satisfies
\begin{align*}
&e^{-\frac{\lambda}{2}(2 \beta_\kappa+(1+\kappa) \frac{c_2 s}{m})}
\quad\leq \quad e^{-\frac{\lambda}{2} s}
\quad\leq\quad e^{\frac{l}{2} \frac{c+1}{m} }
e^{-\frac{\lambda}{2} (s+\frac{c_1 s}{m})}
\end{align*}
because $l \geq \lambda s$ on $E_{l}$,
and because $c_1 \leq c+1$.
In view of this, \eqref{tmp7} entails that the second factor in the 
preceding display is bounded by
$(\delta 2^{m/2+3} \exp( l (c+1) / (2 m)))^{1/2}$, provided that
$C$ is sufficiently large.
In the integrand of the third factor, 
note that the exponential is bounded by $\exp(-\lambda \beta_\kappa)$, 
and that $(c s/m - c_2 s/m) = 2 \sqrt{ c s/m} \|x\|\kappa/(1+\kappa)
	- (\|x\| \kappa/(1+\kappa))^2$
because $(x,s,\lambda)\in D^c_\kappa$.
The arguments used to deal with the corresponding factor
at the end of the proof of Lemma~\ref{N1},
together with Lemma~\ref{bigint},
entail that the third factor in the preceding display
is bounded by a finite constant that we denote by $\Lambda$,
even if the integral is extended over the whole space
$\R^2 \times (0,\infty)^2$.
In summary, we see that 
\begin{equation}\label{boundN22}
N_{\kappa,\epsilon}^{(2,2)} \quad <\quad 
2 \; 
\left(
	\delta 
	2^{\frac{m}{2}+3}
	e^{l\frac{c+1}{2 m}}
	\right)^{\frac{1}{2}} \;
\Lambda
\end{equation}
if $\kappa < \kappa_2(m)$ and  if $C$ is sufficiently large.

Lastly, $N_{\kappa,\epsilon}^{(2,3)}$ is bounded by 
the integral of \eqref{N2.3} over $C^c \cap D_\kappa^c \cap E_{l}^c$.
Upon setting
$$
F(u) \quad=\quad -\frac{2 }{1+\kappa} \lambda^{-1}
	e^{-\frac{\lambda}{2}(2\beta_\kappa + (1+\kappa) \frac{s}{m} u)},
$$
it is elementary to verify that the expression in \eqref{N2.3}
can be written as
\begin{align}\label{N2.4}
&(1+\kappa) \frac{s^{\frac{m}{2}-1} \lambda^{\frac{m}{2}}}{
	2^{\frac{m}{2}+1}}
	\Big(
		F\left( c_3 \right)
		-
		F\left( c_2 \right)
		-
		F'\left( c \right)
		\left(
			c_3 -
			c_2 
		\right)
	\Big) \; \{\lambda>\epsilon\}.
\end{align}
In \eqref{N2.4}, the factor in parentheses involving the function $F$
can be written as
\begin{align*}
&
F\left(  c_3\right) - F\left( c \right)
	-F'\left( c \right)
	\left( c_3 - c \right)
\\
& + \quad
F\left( c \right) - F\left( c_2 \right)
	-F'\left( c \right)
	\left( c - c_2 \right).
\end{align*}
In the preceding display, the expression in the first line
is non-positive, because $F(\cdot)$ is concave. And, again using
the concavity of $F(\cdot)$ and simplifying, 
the expression in the second line
is bounded from above by
\begin{align*}
&
\left(
	F'\left( c_2 \right)
	-F'\left( c  \right)
\right)
	\left( c - c_2 \right)
\quad\leq\quad
- \left( c  - c_2 \right)^2
	F^{''}\left( c_2 \right),
\end{align*}
where the inequality follows from the convexity of $F'(\cdot)$.
Using this to bound \eqref{N2.4} from above, and 
plugging-in the explicit formula for $F^{''}(\cdot)$, we obtain that
$N_{\kappa,\epsilon}^{(2)}$ is bounded from above by
$$
\frac{(1+\kappa)^2}{2^{\frac{m}{2}+2}}
\iiint\limits_{C^c\cap D_\kappa^c \cap E_{l}^c}
s^{\frac{m}{2}-1} \lambda^{\frac{m}{2}+1}
e^{-\frac{\lambda}{2}(2 \beta_\kappa + (1+\kappa)\frac{c_2 s}{m})}
\left( \frac{c s}{m} - \frac{c_2 s}{m}\right)^2
\;\{\lambda > \epsilon\}
\; d \lambda\; d x \; d s.
$$
In the integrand of the third factor, 
note that the exponential is bounded by $\exp(-\lambda \beta_\kappa)$, 
and that $(c s/m - c_2 s/m) = 2 \sqrt{ c s/m} \|x\|\kappa/(1+\kappa)
	- (\|x\| \kappa/(1+\kappa))^2$
because $(x,s,\lambda)\in D^c_\kappa$.
The arguments used to deal with the corresponding factor
at the end of the proof of Lemma~\ref{N1},
Lemma~\ref{bigint}, and the monotone convergence theorem
entail that for each $\delta>0$ we have
\begin{equation} \label{boundN23}
N_{\kappa,\epsilon}^{(2,3)} \quad < \quad \delta
\end{equation}
irrespective of $\kappa$ and $C$, provided that $l$ is sufficiently large,
e.g., $l>l_0(\delta,\epsilon,m,c)$.

To put the pieces together, fix
$l$ sufficiently large 
so that $N_{\kappa,\epsilon}^{(2,3)} < \Delta/18$; cf. \eqref{boundN23}.
Then choose $C$ sufficiently large and $\kappa_2$  sufficiently small,
so that
$N_{\kappa,\epsilon}^{(2,2)} < \Delta/18$ whenever $\kappa < \kappa_2$; 
cf. \eqref{boundN22}.  Lastly, choose $\kappa_1$ sufficiently small, so that
$N_{\kappa,\epsilon}^{(2,1)} < \Delta/18$ whenever $\kappa < \kappa_1$;
cf. \eqref{boundN21}.
	It follows that $N_{\kappa,\epsilon}^{(2)} < \Delta/6$
	provided that $\kappa < \min\{\kappa_1,\kappa_2,\kappa_3\}$
and that $C$ is sufficiently large.
\end{proof}

\begin{lemma}
\label{N3}
Under the assumptions of Proposition~\ref{PRD2} with $p=2$,
we have $N_{\kappa,\epsilon}^{(3)} < \Delta/6$ provided only that
$\kappa$ is sufficiently small and positive
(where $\Delta$ is as in Proposition~\ref{PRDa2}).
\end{lemma}

\begin{proof}
Because $\phi_0(x,s,\mu,\lambda)=0$ if $\mu \in A^{(3)}(x,s)$,
we can write $N_{\kappa,\epsilon}^{(3)}$ as
the integral over $C^c$ of
\begin{align*}
&\int\limits_{ A^{(3)}(x,s)}
\Big( p_{\kappa,\epsilon}(\mu\|x,s,\lambda) 
	- r_\kappa(c s/m|\lambda)\Big)
\phi_1(x,s,\mu,\lambda)\; d \mu
\\
&\quad\quad\times\quad
\; p_{\kappa,\epsilon}(\lambda\|x,s) 
\; q_{\kappa,\epsilon}(x,s).
\end{align*}
In the preceding display, the integrand 
is negative whenever $\|\mu-\mu_\kappa\| >  c s/m$. Therefore,
$N_\kappa^{(3)}$ is bounded from above by
the expression in the preceding display with $A^{(3)}(x,s)$ replaced
by
$B(x,s,\kappa) = \{ \mu:\;  \|\mu-x\|^2 \geq (c+1) s/m \text{ and }
	\|\mu-\mu_\kappa\|^2 \leq c s/m\}$.

Note that $B(x,s,\kappa)$ is empty if
$\|x\|\kappa/(1+\kappa) < h \sqrt{s}$ for some $h>0$.
To see this, suppose that $\mu$ satisfies $\|\mu-\mu_\kappa\|^2 \leq c s/m$.
Then
\begin{align*}
\sqrt{ \frac{c s}{m}} 
&\quad \geq \quad \|\mu - \mu_\kappa\| \quad=\quad \|(\mu-x)+(x-\mu_\kappa)\|\\
&\quad \geq \quad \left|
	\|\mu-x\| - \|x\|\frac{\kappa}{1+\kappa}
	\right|
	\quad \geq \quad 
	\|\mu-x\| - \|x\|\frac{\kappa}{1+\kappa},
\end{align*}
so that
\begin{align*}
\|\mu-x\| 
&\quad \leq \quad \sqrt{\frac{c s}{m}} + \|x\|\frac{\kappa}{1+\kappa} 
\\
&\quad=\quad \sqrt{\frac{(c+1) s}{m}} \;
	\sqrt{\frac{c}{c+1}}\;
	 \left( 1 + 
		\frac{\|x\|}{\sqrt{s}} \frac{\kappa}{1+\kappa} 
		\sqrt{\frac{m}{c}}\right).
\end{align*}
The expression on the far right-hand side of the preceding display
equals $\sqrt{ (c+1) s/m}$ multiplied by a factor that is
smaller than one provided that $( \|x\|/\sqrt{s}) \kappa/(1+\kappa) $
is sufficiently small, e.g., provided that
$\|x\|\kappa/(1+\kappa) < h \sqrt{s}$ for an appropriate positive
constant $h$ that depends only on $c$ and $m$.

In view of the considerations in the two preceding paragraphs, we see
that $N_{\kappa,\epsilon}^{(3)}$ is bounded from above by
\begin{align*}
& \iiint \limits_{ \|x\|\frac{\kappa}{1+\kappa} \geq h \sqrt{s}}
\;\;
\int\limits_{B(x,s,\kappa)}
\Big( p_{\kappa,\epsilon}(\mu\|x,s,\lambda) - 
	r_\kappa(c s/m|\lambda)\Big)
\phi_1(x,s,\mu,\lambda)\; d \mu
\\
&\qquad\qquad
p_{\kappa,\epsilon}(\lambda\|x,s)\;
q_{\kappa,\epsilon}(x,s)\; d\lambda \; d x \; d s
\\
&\quad\leq \quad 
 \iiint \limits_{ \|x\|\frac{\kappa}{1+\kappa} \geq h \sqrt{s}}
2 p_{\kappa,\epsilon}(\mu_\kappa\|x,s,\lambda) 
\;
\pi \frac{ c s}{m} 
\;
p_{\kappa,\epsilon}(\lambda\|x,s)\;
q_{\kappa,\epsilon}(x,s)  
\; d \lambda \; d x \; d s
\\
&\quad = \quad (1+\kappa)\frac{ c}{m} 2^{-\frac{m}{2}}
\iint \limits_{ \|x\|\frac{\kappa}{1+\kappa} \geq h \sqrt{s}}
 s^{\frac{m}{2}} 
\int\limits_\epsilon^\infty 
 \lambda^\frac{m}{2} e^{-\lambda \beta_\kappa}
 \; d\lambda
 \; d x \; d s,
\end{align*}
where the inequality follows upon noting that 
$p_{\kappa,\epsilon}(\mu\|x,s,\lambda)$ is maximized at $\mu=\mu_\kappa$
and hence also that
$r_\kappa(c s/m|\lambda) < 
	p_{\kappa,\epsilon}(\mu_\kappa\|x,s,\lambda)$,
that $\phi_1(x,s,\mu,\lambda)\leq 1$,
and that the area of $B(x,s,\kappa)$ is bounded by $\pi c s/m$,
and where the equality is obtained from the results in
Section~\ref{S4.2} and from elementary simplifications.
It now follows that the expression on the far right-hand side of
the preceding display is smaller that $\Delta/6$ provided only
that $\kappa$ is sufficiently small, by arguing as in the proof
of Lemma~\ref{N2}; cf.  the derivation of the upper bound for 
$N_{\kappa,\epsilon}^{(2,1)}$.
\end{proof}

\begin{proof}[Proof of Proposition~\ref{PRD2}]
Recall that we decomposed the quantity of interest as
$Q_{\kappa,\epsilon}(L_\kappa(\phi_1)) - Q_{\kappa,\epsilon}(L_\kappa(\phi_0)) =
M_{\kappa,\epsilon}-N_{\kappa,\epsilon}^{(1)} - 
	N_{\kappa,\epsilon}^{(2)} - N_{\kappa,\epsilon}^{(3)}$.
The result now follows immediately from the lemmata~\ref{M} 
through \ref{N3}.
\end{proof}

\section{Auxiliary results}
\label{AD}
\label{aux}

Consider a version of
spherical coordinates (cf., say, \citealp{Blu60a}) to re-parameterize
the $(p+1)$-vector
$(t,z')' \in (0,\infty)\times \R^p$ as
$(t,z')' = \varphi(r,\theta_1,\dots, \theta_p)$, where the
$j$-th coordinate the $\varphi(\cdots)$-function is given by
$$
\varphi_j(r,\theta_1,\dots, \theta_p) \quad=\quad
\begin{cases}
r^2 \cos^2 (\theta_1) 
	&\text{ for $j=1$, }\\
r \left(\prod_{i=1}^{j-1} \sin(\theta_i)\right)
	\cos(\theta_j) 
	&\text{ for $1 <  j \leq p$, and}\\
r   \prod_{i=1}^{p} \sin(\theta_i)
	&\text{ for $j=p+1$,}
\end{cases}
$$
and where the arguments of the $\varphi(\cdots)$-function
satisfy $0< r$,
$0\leq \theta_1 < \pi/2$,
$0\leq \theta_i < \pi$ whenever $1< i < p$, and
$0 \leq \theta_p < 2 \pi$.
The determinant of the Jacobian of $\varphi(\cdots)$ is given by
$2 r^{p+1} \cos(\theta_1) \prod_{j=1}^p\sin^{p-j}(\theta_j)$ (which can
be derived by a simple induction).

\begin{lemma} 
\label{bigint}
Fix an integer $p\geq 1$ and  real numbers $\alpha$, $\beta$ and $\gamma$.
If $\alpha + \beta - \gamma + 1 +p/2>0$ and $\alpha+\beta+1 +p/2>0$, then
$$
\int\limits_{\R^p} \int_0^\infty
t^\alpha \|z\|^{2 \beta} \frac{\Gamma(\gamma, t+\|z\|^2)}{
	(t+\|z\|^2)^\gamma} \; d t \; d z
\quad=\quad
\frac{ \pi^{\frac{p}{2}}}{\alpha+\beta-\gamma+1+\frac{p}{2}}
\frac{ \Gamma(\alpha+1) \Gamma(\beta+\frac{p}{2})}{\Gamma(\frac{p}{2})}.
$$
\end{lemma}

\begin{proof}
Using the re-parameterization  $(t,z')'=\varphi(r,\theta_1,\dots,\theta_p)$
introduced earlier, we can write the integral of interest
as
\begin{align*}
& 4 \pi \int_0^\infty 
	r^{2 \alpha+2\beta-2\gamma+p+1} 
	\Gamma\left(\gamma, r^2\right) \; d r
\quad\times\quad
\int_{0}^{\frac{\pi}{2}}
	\sin^{2\beta+p-1}(\theta_1) \cos^{2 \alpha+1}(\theta_1)
	\; d \theta_1 
\\
&\quad \times\quad
\prod_{j=2}^{p-1}\int_{0}^{\pi} \sin^{p-j}(\theta_j)\; d\theta_j.
\end{align*}
In the preceding display, the first integral (with respect to  $r$)
can be computed as 
$2 \Gamma(\alpha+\beta+1+p/2)/(\alpha+\beta-\gamma+1+p/2)$
using integration by parts.
For the remaining integrals, we repeatedly use the identity
$\int_{0}^{\pi/2} \sin^{a-1}(\theta) \cos^{b-1}(\theta) d \theta
= B(a/2, b/2)/2$, which holds provided that $a>0$ and $b>0$;
cf. \citet[Relation 3.621.5]{Gra80c}.
With this, the result follows after elementary simplifications.
\end{proof}

\begin{lemma}\label{lemmaD} 
Fix an integer $p\geq 1$ and a real number $\gamma$.
If $\gamma>(p-2)/2$, then
$$
\lim_{\delta\downarrow 0}
\delta^{\frac{p-2}{2}-\gamma}
\iint \limits_{\substack{ \R^p \; \;(0,\infty)\\ 
			\|z\|^2 \delta > t}}
			t^{\gamma-\frac{p}{2}} 
			\frac{\Gamma(\gamma,t+\|z\|^2)}{
			(t+\|z\|^2)^\gamma}
			\; d t \;d z
\quad=\quad
\frac{2\pi^\frac{p}{2}}{2\gamma+2-p}\frac{\Gamma(\gamma+1)}{\Gamma(p/2)}.
$$
\end{lemma}

\begin{proof}
Re-parameterizing $(t,z')'$ as $\varphi(r,\theta_1,\dots, \theta_p)$,
and noting that the condition $\|z\|^2 \delta > t$
can be re-expressed as $\theta_1 > \arctan( \delta^{-1/2})$,
we can write the integral as
\begin{align*}
&
4 \pi \int_0^\infty r^{2\gamma+1} \frac{\Gamma(\gamma,r^2)}{r^{2\gamma}} 
	\; d r
\quad\times \quad \int_{\arctan(\delta^{-\frac{1}{2}})}^{\frac{\pi}{2}} 
	\sin^{p-1}(\theta_1)
	\cos^{2\gamma+1-p}(\theta_1) \; d \theta_1
\\
&\quad \times \quad \prod_{j=2}^{p-1} \int_0^\pi \sin^{p-j}(\theta_j)\; 
	d\theta_j.
\end{align*}
Arguing as in the proof of Lemma~\ref{bigint}, we see that,
in the preceding display, the first integral equals
$\Gamma(\gamma+1)/2$ and the integral corresponding to the $j$-th
term in the product equals $B((p-j+1)/2, 1/2)$. If we can show
that the second integral divided by $\delta^{\gamma-(p-2)/2}$ converges
to $1/(2\gamma+2-p)$, the result follows after elementary simplifications.

To deal with the second integral in the preceding display, note that
its integrand is bounded from above
by $(\pi/2 - \theta_1)^{2\gamma+1-p}$; and for each $\epsilon>0$
that integrand is bounded from below by
$(1-\epsilon) (\pi/2 - \theta_1)^{2\gamma+1-p}$ provided that
$\theta_1$ is sufficiently close to $\pi/2$, i.e.,  provided that
$\delta$ is sufficiently small.
In view of this, the result follows upon noting that
\begin{align*}
\lim_{\delta\downarrow 0} \delta^{\frac{p-2}{2}-\gamma} 
\int \limits_{\arctan(\delta^{-\frac{1}{2}})}^{\frac{\pi}{2}}
	\left(\frac{\pi}{2} - \theta_1\right)^{2\gamma+1-p} \; d \theta_1
\quad=\quad \frac{1}{2\gamma+2-p}.
\end{align*}
\end{proof}

\begin{lemma} \label{smoments} 
In the setting of Section~\ref{S3}, we have
$$
P_{\kappa,\epsilon}( s^{p/2}) \quad=\quad 
\frac{1}{2} \left(\frac{\epsilon}{2}\right)^{-\frac{p}{2}}
\frac{ \Gamma((m+p)/2)}{\Gamma(p/2)}
$$
for each $\kappa>0$.
\end{lemma}

\begin{proof}
Since $s \| \lambda \sim \lambda^{-1} \chi^2_m$, it is easy to see
that $P_{\kappa,\epsilon}( s^{p/2} \| \lambda)$, i.e., 
the conditional mean of $s^{p/2}$
given $\lambda$ under $P_{\kappa,\epsilon}$, equals 
$\lambda^{-p/2} 2^{p/2} \Gamma((m+p)/2)/\Gamma(m/2)$.
And using the marginal density of $\lambda$ under $P_{\kappa,\epsilon}$,
cf. Proposition~\ref{propNtG1},
it is elementary to verify that $P_{\kappa,\epsilon}(\lambda^{-p/2})$
equals $\epsilon^{-p/2}/2$.
\end{proof}

\end{appendix}

{\small
\setlength{\parskip}{0.5 \medskipamount}
\bibliographystyle{plainnat}
\bibliography{lit}
}

\end{document}